\documentclass[12pt, english, a4paper]{article}
\usepackage[latin9]{inputenc}
\usepackage[T1]{fontenc}
\usepackage{graphicx}
\usepackage{amsthm}
\usepackage{amsfonts}
\usepackage{mathtools}
    \allowdisplaybreaks
\usepackage{amssymb}
\usepackage{makecell}
\usepackage{comment}
\usepackage{cases}
\usepackage{mathrsfs}
\usepackage{fullpage}
\usepackage{times}
\usepackage[usenames]{color}
\usepackage[dvipsnames]{xcolor}
\usepackage{hyperref}
\usepackage{tikz,tikz-cd}
\usepackage{extarrows}

\usepackage{ytableau}
\usepackage{xcolor}
\definecolor{lgrey}{rgb}{0.8,0.8,0.8 }
\definecolor{dgrey}{rgb}{0.3,0.3,0.3 }

\newcommand{\SSS}{\mathfrak{S}}

\newcommand{\card}[1]{\lvert #1 \rvert}
\newcommand{\N}{\mathbb{N}}
\newcommand{\Av}{\mathfrak{S}}
\newcommand{\sq}{\square}
\DeclareMathOperator{\st}{st}
\theoremstyle{plain}
\newtheorem{theorem}{Theorem}

\newtheorem{lemma}[theorem]{Lemma}

\theoremstyle{definition}

\newtheorem{example}[theorem]{Example}

\theoremstyle{remark}
\newtheorem{remark}[theorem]{Remark}

\usepackage{authblk}
\title{Ordinal and disjoint sums of partially ordered patterns}
\author[1]{Sucharita Biswas}
\author[2]{Umesh Shankar\thanks{\tt{204093001@iitb.ac.in, umeshshankar@outlook.com}}} 
\author[3]{Sivaramakrishnan Sivasubramanian}
\affil[1,2,3]{Department of Mathematics, Indian Institute of Technology, Bombay Mumbai 400076, India} 
\date{\today}
\begin{document}
\maketitle
\begin{abstract}
Partially ordered patterns (POPs) generalize the classical notion of permutation patterns within the framework of pattern avoidance. Building on recent work by Burstein, Han, Kitaev, and Zhang, which introduced the concept of shape-Wilf-equivalence of sets of patterns, we develop the notions of \emph{ordinal} and \emph{disjoint sums} of labeled posets. This framework enables us to reinterpret their main result as an ordinal sum analogue of the classical theorem by Backelin, West, and Xin. We establish analogous results for disjoint sums of POPs and further extend their results to prove Wilf-equivalence for classes of POPs that include isolated vertices. In particular, we prove the shape-Wilf-equivalence of the sets of patterns $\{123, 213, 312\}$ and $\{132, 231, 321\}$. Our proof strategy involves a bijection that filters through an encoding scheme for the transversals avoiding these patterns.  We use these results to completely classify the partially ordered patterns of size $3,4,5$ whose connected components are all chains. This classification also confirms a conjecture posed by Dimitrov at the problem session of the British Combinatorics Conference 2024 (BCC30).
\end{abstract}
\textbf{\small{}Keyword:}{\small{} partially ordered pattern, shape-Wilf-equivalence, Wilf-equivalence, poset operation, isolated vertex  }{\let\thefootnote\relax}
\section{Introduction}

Let $\N$ denote the set of natural numbers. Let $[n] \coloneqq \{1,\dots,n\}$ and $\SSS_n$ denote the set of permutations of $[n]$. 
An occurrence of a (classical) permutation pattern $p=p_1,p_2,\ldots, p_k$ in a permutation $\pi=\pi_1\dots\pi_n \in \SSS_n$ is a subsequence $\pi_{i_1}\dots\pi_{i_k}$ that is order isomorphic to $p$ (i.e. $\pi_{i_j}<\pi_{i_m}$ if and only if $p_j<p_m$). Permutation patterns are very well studied and have many connections to computer science, algebra, geometry and mathematical biology (see the text by Kitaev \cite{kitaev-patterns-words} and the references therein).   

Partially order patterns were introduced by Kitaev \cite{kitaev-pop} and have been extensively studied in words and compositions (see, for example, \cite{kitaev-segmented-pop}, \cite{kitaev-mansour-pops-kary}, \cite{gao-kitaev-length-4-5-pops}).
A \emph{partially ordered pattern} (POP) $p$ of size $k$ is defined by a $k$-element partially ordered set (poset) $P$ labeled by the elements in $\{1,\dots, k\}$. An occurrence of such a POP $p$ in a permutation $\pi = \pi_1\dots\pi_n$ is a subsequence $\pi_{i_1}\dots \pi_{i_k}$, where $1 \le i_1 <\dots < i_k \le n$, such that $\pi_{i_j} < \pi_{i_m}$ if and only if $j < m$ in $P$. We will use the word ``pattern'' to mean a classical or partially ordered pattern. A permutation $\pi$ is said to avoid a pattern $p$ if it does not contain any occurrences of the pattern $p$. We denote by $\SSS_n(p)$ the set of permutations in $\SSS_n$ that avoid the pattern $p$. We say that two pattern $p,q$ are \emph{Wilf-equivalent}, denoted by $p \sim q$, if $\card{\SSS_n(p)}=\card{\SSS_n(q)}$ for all $n\ge 1$.
Alternatively, from the work of Bean et al \cite{bean-bipartite-pop}, one can define the avoidance of a POP $p$ as a simultaneous avoidance of the set of patterns $S$ that present as (group-theoretic) inverses of linear extensions of $p$. Therefore, we will not distinguish between a POP and the set of patterns that generate $\SSS_n(p)$.

In the following, we review the definition of  the notion of shape-Wilf-equivalence for POPs. Let us first review some terminology related to pattern avoiding transversals. Throughout the paper, we draw Ferrers boards in French notation (rows increasing in length from top to bottom), and number columns from left to right and rows from bottom to top. We use the square $(i,j)$ to represent the square located in row $i$ and column $j$. A transversal of a Ferrers board $\lambda= (\lambda_1,\lambda_2,\dots,\lambda_n)$ with $\lambda_1 \ge \lambda_2 \ge \dots \ge \lambda_n > 0$ is a $01$-filling of the squares of $\lambda$ with $1$'s and $0$'s such that every row and every column contains exactly one $1$. A permutation $\pi=\pi_1\pi_2\dots\pi_n$ can be regarded as a transversal of the $n$ by $n$ square diagram, in which the square $(\pi_i,i)$ is filled with a $1$ for all $1 \le i\le n$ and all the other squares are filled with $0$'s. The transversal corresponding to the permutation $\pi$ is also called the permutation matrix
of $\pi$. Let $\SSS_\lambda$ denote the set of transversals of the Ferrers board $\lambda$.

For example, let us consider the Ferrers board $\lambda=(5,5,4,3,3)$ and the permutation $\pi=4 5 2 3 1$. Then $\pi$ can be regarded as the transversal 
\begin{center}
\ytableausetup{centertableaux} \begin{ytableau}
    $0$ & 1 & 0\\
     1 & 0 & 0  \\
    0 & 0 & 0 & 1 \\
    0 & 0 & 1 & 0 & 0\\
    0 & 0 & 0 & 0 & 1\\
\end{ytableau}.
\end{center}
\vspace{.1 in}

Given a permutation $\alpha$ of $\SSS_m$, let $M$ be its permutation matrix. A transversal $T$
of a Ferrers board $\lambda$ with $n$ columns and $n$ rows will be said to contain the pattern
$\alpha$ if there exists two subsets of the index set $[n]$, namely, $R= \{r_1,r_2,\dots,r_m\}$ and
$C= \{c_1,c_2,\dots,c_m\}$, such that the matrix $M'$ on the set $R$ of rows and the set $C$
of columns is a copy of M and each of the squares $(r_i,c_j)$ falls within the Ferrers board. In this context, we say that the matrix $M'$ is isomorphic to $\alpha$ or the matrix
$M'$ is an occurrence of $\alpha$. Otherwise, we say that $T$ avoids the pattern $\alpha$ and $T$
is $\alpha$-avoiding. Given a set $P$ of patterns, let $\SSS_\lambda(P)$ denote the set of transversals of the Ferrers board $\lambda$ that
avoid each pattern in $P$. Given two sets $P$ and $Q$ of patterns, we say that $P$ and $Q$ are shape-Wilf-equivalent if $\card{\SSS_\lambda(P)}=\card{\SSS_\lambda(Q)}$ holds for any Ferrers board $\lambda$. In this context, we write $P \sim_s Q$. If $\{\sigma\}\sim_s \{\tau\}$, we simply write $\sigma\sim_s \tau$. Clearly, the shape-Wilf-equivalence would imply the Wilf-equivalence.

We define the ordinal sum of two labeled posets. First, we recall the ordinal sum of two posets. Let $p=(T,\le_p)$ and $q=(T',\le_q)$ be two posets. The \emph{ordinal sum} $p\oplus q$ is the relation on $T\bigsqcup T'$ that satisfies 
\begin{enumerate}
\item  if $x,y \in T$ such that $x\le_p y$, then $x\le y$ in $p\oplus q$; 
\item  if $x,y\in T'$ such that $x\le_q y$, then $x\le y$ in $p\oplus q$; 
\item  if $x\in T$ and $y\in T'$, then $x\le y$ in $p\oplus q$.
\end{enumerate}

To extend this notion to labeled posets, we specify what happens to the labels of $p$, $q$ in $p\oplus q$. The labels $L_p \subset \mathbb Z_{>0}$ of the poset $T$ in $p\oplus q$ are unchanged while the labels $L_q \subset \mathbb{Z}_{\ge 0}$ of  the poset $T'$ in $p \oplus q$ are all raised by $\max(L_p)$ , that is, $i$ is replaced by $i+\max(L_p)$ for all $i\in L_q$.  
In this work, the labeled posets will have labels from $\mathbb{Z}_{>0}$ and no two vertices will have the same label.

For example, if 
$p=$
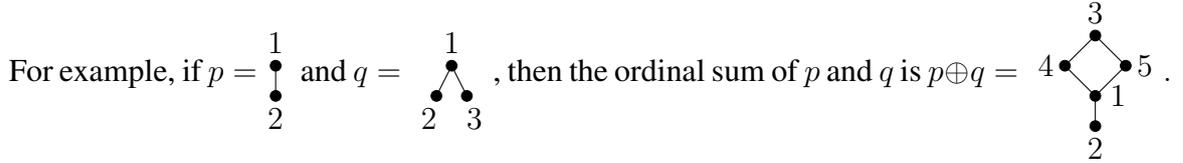
\begin{tikzpicture}[baseline=(current bounding box)]
		\filldraw  (0,.2) circle (2pt);
        \filldraw  (0,-.2) circle (2pt);
        
		\draw (0,.2) -- (0,-.2); 
        
        \node[] at (0,.5)   { $1$};
		\node[] at (0,-.5)  { $2$};
		
\end{tikzpicture}  and $q=\begin{tikzpicture}[baseline=(current bounding box)]
		\filldraw  (0,.2) circle (2pt);
        \filldraw  (-.2,-.2) circle (2pt);
        \filldraw  (.2,-.2) circle (2pt);
		
		\draw (0,.2) -- (-.2,-.2); 
        \draw (0,0.2) -- (0.2,-.2); 
        \node[] at (0,.5)   { $1$};
		\node[] at (-.3,-.5)  { $2$};
		\node[] at (.3,-.5) { $3$};
        
\end{tikzpicture}$, then the ordinal sum of $p$ and $q$ is $p\oplus q=
\begin{tikzpicture}[baseline=(current bounding box)]
        \filldraw  (0,0) circle (2pt);
		\filldraw  (0.4,.4) circle (2pt);
        \filldraw  (-.4,.4) circle (2pt);
        \filldraw  (0,.8) circle (2pt);
        \filldraw  (0,-0.4) circle (2pt);
        
		\draw (0.4,.4) -- (0,0); 
        \draw (-.4,.4) -- (0,0); 
        \draw (0.4,.4) -- (0,.8); 
        \draw (-.4,0.4) -- (0,.8); 
        \draw (0,-0.4) -- (0,0); 
        
        \node[] at (0.3,0)   { $1$};
		\node[] at (-.65,.4)  { $4$};
		\node[] at (.65,.4) { $5$};
        \node[] at (0,1.1)   { $3$};
        \node[] at (0,-.7)   { $2$};
        
\end{tikzpicture}.$

Note that for a POP $p$ on $k$ elements, the set of labels of $p$ will be $\{1,2,\dots, k\}$.
In this language, Theorem $2.2$ of \cite{burstein-shape-wilf} can be recast as follows.

\begin{theorem}{\cite[Theorem 2.2]{burstein-shape-wilf}}\label{thm: ordinal sum}
    Let $p, p', q$ be POPs with $p \sim_s p'$ (i.e. $p$ is shape-Wilf equivalent
    to $p'$). Then, 
    the POPs $p \oplus q$ and $p' \oplus q$ are shape-Wilf-equivalent i.e. $p\oplus q \sim_s p' \oplus q$. 
\end{theorem}
The result of Burstein et al. is a POP analogue of the famous result of Backelin, West and Xin \cite[Proposition 2.3]{bwx-main}.
 Our first main result of this work is a Wilf-equivalence analogue of Theorem \ref{thm: ordinal sum} for the case when there are isolated vertices present.

\begin{theorem}
\label{thm: main1-shape}
    Let $p,p'$ be two POPs of size $k$ with isolated vertices $I=\{i_1< i_2<\dots< i_s\}$  such that the vertices labeled $[i_1,k]\backslash I$ are greater than the vertices labeled $[i_1-1]$, that is, if $x \in [i_1,k]\backslash I$ and $y \in [i_1-1]$, then $x>_p y$ and $x >_{p'}y$ . Let $I_p,I_{p'}$ be the subposets induced by $[i_1-1]$ and $J_p,J_{p'}$ be the subposets induced by $[i_1,k]\backslash I$. If $I_p\sim_s I_{p'}$ and $J_p=J_{p'}$, then $p\sim p'$.
\end{theorem}

To illustrate this theorem, consider the POPs $p=$
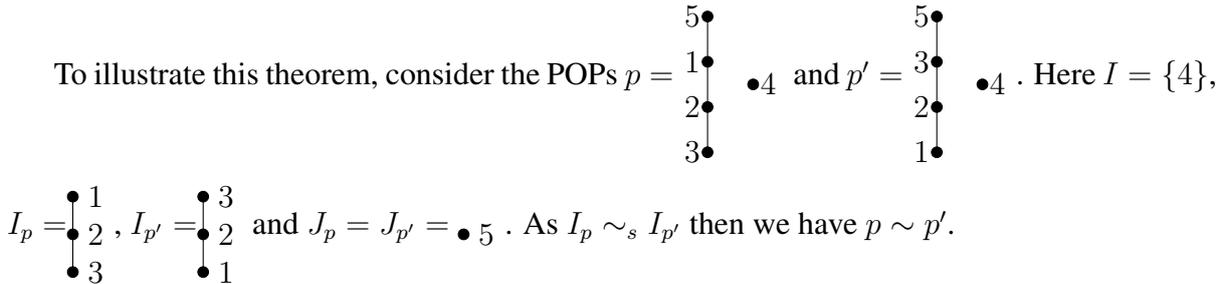
\begin{tikzpicture}[baseline=(current bounding box)]
		\filldraw  (0,.9) circle (2pt);
        \filldraw  (0,.3) circle (2pt);
        \filldraw  (0,-.3) circle (2pt);
        \filldraw  (0,-.9) circle (2pt);
        \filldraw  (.6,0) circle (2pt);
		
		\draw (0,.9) -- (0,.3); 
        \draw (0,0.3) -- (0,-.3); 
        \draw (0,-0.3) -- (0,-.9); 
        \node[] at (-.2,0.9)  {$5$};
		\node[] at (-.2,.3)  { $1$};
		\node[] at (-.2,-.3)  {$2$};
        \node[] at (-.2,-.9)  {$3$};
        \node[] at (.8,0)    { $4$};
		
	\end{tikzpicture} and $p'=$\begin{tikzpicture}[baseline=(current bounding box)]
		\filldraw  (0,.9) circle (2pt);
        \filldraw  (0,.3) circle (2pt);
        \filldraw  (0,-.3) circle (2pt);
        \filldraw  (0,-.9) circle (2pt);
        \filldraw  (.6,0) circle (2pt);
		
		\draw (0,.9) -- (0,.3); 
        \draw (0,0.3) -- (0,-.3); 
        \draw (0,-0.3) -- (0,-.9); 
        \node[] at (-.2,0.9)  {$5$};
		\node[] at (-.2,.3)  { $3$};
		\node[] at (-.2,-.3)  {$2$};
        \node[] at (-.2,-.9)  {$1$};
        \node[] at (.8,0)    { $4$};
		
	\end{tikzpicture}. Here $I=\{4\}$, $I_p=$\begin{tikzpicture}[baseline=(current bounding box)]
    \filldraw  (0,0) circle (2pt);
    \filldraw  (0,-.5) circle (2pt);
    \filldraw  (0,0.5) circle (2pt);
    \draw (0,0) -- (0,-.5); 
    \draw (0,0) -- (0,.5);
    \node[] at (0.3,0) {$2$};
    \node[] at (0.3,.5)   { $1$};
    \node[] at (0.3,-.5)   { $3$};
\end{tikzpicture}, $I_{p'}=$\begin{tikzpicture}[baseline=(current bounding box)]
    \filldraw  (0,0) circle (2pt);
    \filldraw  (0,-.5) circle (2pt);
    \filldraw  (0,0.5) circle (2pt);
    \draw (0,0) -- (0,-.5); 
    \draw (0,0) -- (0,.5);
    \node[] at (0.3,0) {$2$};
    \node[] at (0.3,.5)   { $3$};
    \node[] at (0.3,-.5)   { $1$};
\end{tikzpicture} and $J_p=J_{p'}= $ \begin{tikzpicture}[baseline=(current bounding box)]
    \filldraw  (0,-0.2) circle (2pt);
    \node[] at (0.3,-0.2) {$5$};
\end{tikzpicture}. As $I_p \sim_s I_{p'}$ then we have $p \sim p'$.



Our second main result is the Wilf-equivalence of certain POPs with a specific form.

		
		
		
		

\begin{theorem}
\label{thm: main2-wilf-3size}
    Let $p,p'$ be two POPs of size $k$ with the same isolated vertex labels $I$ such that $k-3\notin I, k-1\in I$ and  the vertices labeled $k-2,k$ are greater (in $p,p'$) than the vertices labeled $[k-3]\backslash I$, that is, if $x \in [k-3]\backslash I$, then $k-2,k \geq_p x$ and $k-2,k \geq_p' x$.. Let $I_p, I_{p'}$ be the subposets induced by the vertices labeled $[k-3]\backslash I$ and $J_p,J_{p'}$ be the subposets induced by $[k-2,k]\backslash I$ in $p,p'$ respectively. 
    If $I_p=I_{p'}$ and $$J_p = \begin{tikzpicture}[baseline=(current bounding box)]
		\filldraw  (0,0) circle (2pt); 
		\filldraw   (0,.5) circle (2pt);
        		\draw (0,0) -- (0,.5); 
        \node[] at (0,.75)  { $k-2$};
		\node[] at (0,-.3)  { $k$};
\end{tikzpicture}, \hspace{1cm} J_{p'}=\begin{tikzpicture}[baseline=(current bounding box)]
		\filldraw  (0,0) circle (2pt); 
		\filldraw   (0,.5) circle (2pt);
        		\draw (0,0) -- (0,.5); 
        \node[] at (0,.8)  { $k$};
		\node[] at (0,-.25)  { $k-2$};
\end{tikzpicture},$$  then $p \sim p'$.
\end{theorem}

We illustrate Theorem \ref{thm: main2-wilf-3size} through an example. Consider the POPs $p=$\begin{tikzpicture}[baseline=(current bounding box)]
		\filldraw  (0,.9) circle (2pt);
        \filldraw  (0,.3) circle (2pt);
        \filldraw  (0,-.3) circle (2pt);
        \filldraw  (0,-.9) circle (2pt);
        \filldraw  (.6,0) circle (2pt);
		
		\draw (0,.9) -- (0,.3); 
        \draw (0,0.3) -- (0,-.3); 
        \draw (0,-0.3) -- (0,-.9); 
        \node[] at (-.2,0.9)  {$3$};
		\node[] at (-.2,.3)  { $5$};
		\node[] at (-.2,-.3)  {$1$};
        \node[] at (-.2,-.9)  {$2$};
        \node[] at (.8,0)    { $4$};
		
	\end{tikzpicture} and $p'=$\begin{tikzpicture}[baseline=(current bounding box)]
		\filldraw  (0,.9) circle (2pt);
        \filldraw  (0,.3) circle (2pt);
        \filldraw  (0,-.3) circle (2pt);
        \filldraw  (0,-.9) circle (2pt);
        \filldraw  (.6,0) circle (2pt);
		
		\draw (0,.9) -- (0,.3); 
        \draw (0,0.3) -- (0,-.3); 
        \draw (0,-0.3) -- (0,-.9); 
        \node[] at (-.2,0.9)  {$5$};
		\node[] at (-.2,.3)  { $3$};
		\node[] at (-.2,-.3)  {$1$};
        \node[] at (-.2,-.9)  {$2$};
        \node[] at (.8,0)    { $4$};
		
	\end{tikzpicture}. Here $I=\{4\}$, $I_p=I_{p'}= \begin{tikzpicture}[baseline=(current bounding box)]
		\filldraw  (0,0) circle (2pt); 
		\filldraw   (0,.5) circle (2pt);
        		\draw (0,0) -- (0,.5); 
        \node[] at (.2,.7)  { $1$};
		\node[] at (.2,-.2)  { $2$};
\end{tikzpicture}$. Hence, by Theorem \ref{thm: main2-wilf-3size}, we get Wilf-equivalence of $p$ and $p'$.

We give the proof of Theorems \ref{thm: main1-shape} and \ref{thm: main2-wilf-3size} in Section \ref{sec: proofs}. Similar to the study of the effect of ordinal sums on pattern avoidance, we study the effect of another poset operation on pattern avoidance. Given $p=(T,\le_p)$ and $q=(T',\le_q)$, the \emph{disjoint sum} $p+q$ is the relation on $T\bigsqcup T'$ defined by
\begin{enumerate}
    \item if $x,y\in T$ such that $x\le_p y$, then $x\le y$ in $p+q$;
    \item if $x,y$ in $T'$ such that $x\le_q y$, then $x\le y$  in $p+q$.
\end{enumerate}

To extend this to labeled posets, we specify how the labels are handled. The labels $L_p$ in $T$ are untouched, while the labels $L_q$ of $T'$ in $p+q$ are all raised by $\max(L_p)$.

For instance, let $p=$\begin{tikzpicture}[baseline=(current bounding box)]
		\filldraw  (0,.2) circle (2pt);
        \filldraw  (0,-.2) circle (2pt);
        
		\draw (0,.2) -- (0,-.2); 
        
        \node[] at (0,.5)   { $1$};
		\node[] at (0,-.5)  { $2$};
		
\end{tikzpicture} and $q=\begin{tikzpicture}[baseline=(current bounding box)]
		\filldraw  (0,.2) circle (2pt);
        \filldraw  (-.2,-.2) circle (2pt);
        \filldraw  (.2,-.2) circle (2pt);
		
		\draw (0,.2) -- (-.2,-.2); 
        \draw (0,0.2) -- (0.2,-.2); 
        \node[] at (0,.5)   { $1$};
		\node[] at (-.3,-.5)  { $2$};
		\node[] at (.3,-.5) { $3$};
        
\end{tikzpicture}$. The disjoint sum of $p$ and $q$ is $p+q= \begin{tikzpicture}[baseline=(current bounding box)]
		\filldraw  (0,.2) circle (2pt);
        \filldraw  (-.2,-.2) circle (2pt);
        \filldraw  (.2,-.2) circle (2pt);
        \filldraw  (-.7,0.2) circle (2pt);
        \filldraw  (-.7,-.2) circle (2pt);
		
		\draw (0,.2) -- (-.2,-.2); 
        \draw (0,0.2) -- (0.2,-.2); 
        \draw (-.7,0.2) -- (-.7,-.2); 
        \node[] at (0,0.5)   { $3$};
		\node[] at (-.7,-.5)  { $2$};
		\node[] at (.35,-.5) { $5$};
        \node[] at (-.7,.5)   { $1$};
        \node[] at (-.35,-.5)   { $4$};
        
\end{tikzpicture}$.

We state our theorem concerning Wilf-equivalence of disjoint sums of POPs.

\begin{theorem}\label{thm: disjoint-sum-shape-w}
    If $p\sim p'$ and $q\sim q'$, then $p+q\sim p'+q'$.
\end{theorem}

Our next theorem states that, up to Wilf-equivalence, the patterns $p+q$ and $q+p$ are equivalent.

\begin{theorem}\label{thm: disjoint-sum-wilf}
    Let $p, q$ be two POPs. Then, $p+q\sim q+p$.
\end{theorem}



Our proofs of Theorems \ref{thm: disjoint-sum-shape-w} and \ref{thm: disjoint-sum-wilf} appear in Section \ref{sec: proofs-3-4}.
Using Theorems \ref{thm: main1-shape}, \ref{thm: main2-wilf-3size}, \ref{thm: disjoint-sum-shape-w}, \ref{thm: disjoint-sum-wilf}, we classify the Wilf-equivalence classes of POPs of size $3,4,$ and $5$ all of whose connected components are linearly ordered (chains). The classification of size $3$ POPs are done in Section \ref{sec: len3}. In this section, we additionally show the following shape-Wilf-equivalence result.
\begin{theorem}\label{thm: len3-B12-B21}
    The POPs $p:=$\begin{tikzpicture}[baseline=(current bounding box)]
    \filldraw  (1,0.25) circle (2pt);
    \filldraw  (1,-.25) circle (2pt);
    \filldraw  (0.5,0) circle (2pt);
    \draw (1,.25) -- (1,-.25); 
    \node[] at (0.3,0) {$1$};
    \node[] at (1.3,.25)   { $2$};
    \node[] at (1.3,-.25)   { $3$};
\end{tikzpicture}
and $p':=$\begin{tikzpicture}[baseline=(current bounding box)]
    \filldraw  (1,0.25) circle (2pt);
    \filldraw  (1,-.25) circle (2pt);
    \filldraw  (0.5,0) circle (2pt);
    \draw (1,.25) -- (1,-.25); 
    \node[] at (0.3,0) {$1$};
    \node[] at (1.3,.25)   { $3$};
    \node[] at (1.3,-.25)   { $2$};
\end{tikzpicture} are shape-Wilf-equivalent.
\end{theorem}

In Section \ref{sec: len4}, we completely classify Wilf-equivalence classes,
for POPs of size $4$ whose components are all chains. 
Finally, in Section \ref{sec: len5}, we completely classify the 
Wilf-equivalence classes for POPs of size $5$ whose 
components are all chains.

\section{Proofs of Theorems \ref{thm: main1-shape} and \ref{thm: main2-wilf-3size}}\label{sec: proofs}

To prove Theorem \ref{thm: main1-shape}, we follow an idea presented in many sources (see \cite{bwx-main}, \cite{baxter-shape-wilf-vincular}, \cite{baxter-even-perm}, \cite{burstein-shape-wilf}). The key difference is presented in the next lemma.  

Before we jump into the proofs, we define the standardisation of a labeled poset. Let $P$ be a labeled poset with labels $x_1,\dots,x_r\in \mathbb{Z}_{>0}$ such that $x_1<\dots<x_r$. Then, the standardisation of $P$ is the labeled poset $\st(P)$ with the same underlying vertices but the labels $x_i$ are replaced by $i$ for each $i\in [r]$. Note that standardisation makes any labeled poset a POP.

\begin{lemma}\label{lem: isolated-occ}
     Let $p$ be a POP of size $r$ and $I=\{i_1,\dots, i_s\}$ be the set of isolated vertices. An occurrence of $p$ in a transversal of a square Ferrers board $F$ is equivalent to the existence of two subsets $R=\{x_1,\dots,x_{r-s}\}$ and $C=\{ y_1,\dots, y_r \}$ such that the submatrix 
    $$\begin{pmatrix}
        F[x_1,y_1]&F[x_1,y_2]&\dots&F[x_1,y_r]\\
        F[x_2,y_1]&F[x_2,y_2]&\dots&F[x_2,y_r]\\
        \vdots&\vdots&\ddots &\vdots\\
        F[x_{r-s},y_1]&F[x_r,y_2]&\dots&F[x_{r-s},y_r]
    \end{pmatrix}$$
    contains an occurrence of $\st(p\backslash I)$ in the columns indexed by $[r]\backslash I$ and each cell $(x_i,y_j)$ lies inside $F$ for $1\le i\le r-s$ and $1\le j\le r$.
\end{lemma}
\begin{proof}
Let $M$ be a submatrix that is an occurrence of $p$ in $F$.  Let $M'$ be the matrix obtained by deleting from $M$ the rows that contains $1$'s in the columns indexed by $I$. 

In $M$, the columns indexed by $[r]\backslash I$ contain an occurrence of $\st(p\backslash I)$. By deleting these rows that contains $1$'s in the columns indexed by $I$, we have not affected ones that appear in the columns in $[r]\backslash I$. Therefore, the columns indexed by $[r]\backslash I$ in $M'$ contain an occurrence of $\st(p\backslash I)$.

Similarly, given a $(r-s) \times s$ submatrix such that the columns indexed by $[r]\backslash I$ contain an occurrence of $\st(p\backslash I)$, we can look for the rows that contain occurrence of $1$'s in the columns indexed by $I$ and add the corresponding cells to the matrix $M'$ to uniquely recover $M$. Since $F$ is a square Ferrers board, this newly constructed $M$ will lie completely inside $F$, thereby completing the proof.
\end{proof}
We call such a $(r-s)\times s$ submatrix, obtained in Lemma \ref{lem: isolated-occ}, \emph{an essential occurrence} of $p$.
We are now ready to prove Theorem \ref{thm: main1-shape}.

\begin{proof}[Proof of Theorem \ref{thm: main1-shape}]
Suppose $F$ is a square Ferrers board and a transversal $f\in F(p)$. Fix a bijection $g$ that maps $\Av_F(I_p)$ to $\Av_F(I_p')$. This exists as these patterns are shape-Wilf-equivalent. We start with all the squares colourless. 
\begin{enumerate}
\item For any square $(i,j)\in F$, if the subboard of $F$ that is above and to the right of $(i,j)$ contains an essential occurrence of $J_p$ (the submatrix lies entirely to the top-right of the cell), then we colour it white; otherwise, we colour it gray.

\item Find the 1's coloured gray, and colour the corresponding rows and columns gray.

\item  Remove the gray squares from $F$, denote by $F'$ the shape formed by the white squares, justified (squashed) bottom-left, and denote by $f'$ the induced transversal of $F'$. Note that $F'$ is a Ferrers board and the transversal $f'$ avoids $J_p$.  

\item Since $I_p \sim_s I_{p'}$, we can map $f'$ into $f''$ bijectively, using $g$ which we know exists, having the same shape and avoiding $I_{p'}$. 

\item Put the gray squares back in their original places, bijectively, to get a transversal
$f'''\in F(p')$.
\end{enumerate}
Therefore, we have $p\sim p'$.  The proof is complete.
\end{proof}

Before we proceed to our proof of Theorem \ref{thm: main2-wilf-3size}, we 
need a couple of definitions. Let $p$ be a POP of size $k$. 
For $1\le r\le k$, define $p_r$ to be the labeled subposet induced by the vertices labeled $[r]$. A \emph{prefix} of the POP $p$ is an occurrence of the POP $p_r$ for some $1\le r\le k$. 

Given a POP $p$ and a permutation $w$, define $\emph{p-rank}$ of an entry in the permutation $w$ to be the longest $w_{i_1}\dots w_{i_k}$ that is a prefix of the POP $p$. For a permutation in the set $\Av_n(p)$ where $p$ is a POP of size $k$, the $\emph{p-rank}$ of an entry can range anywhere between $1$ and $k-1$.
The following lemma can be proved using Theorem \ref{thm: main1-shape} by the reversal operation but we give a different proof.

\begin{lemma}\label{lem: from-west}
Let $p,p'$ be two POPs of size $k$ with the same isolated vertex labels $I$ such that $k-2\notin I$ and the vertices labeled $k-1,k$ are greater than the vertices labeled $[k-2]\backslash I$ in both $p,p'$, that is, if $x \in [k-2]\backslash I$, then $k-1,k \geq_p x$ and $k-1,k \geq_p' x$. Let $I_p, I_{p'}$ be the subposets induced by the vertices labeled $[k-2]\backslash I$ and $J_p,J_{p'}$ be the subposets induced by ${k-1,k}$ in $p,p'$ respectively. 
    If $I_p=I_{p'}$ and $$J_p = \begin{tikzpicture}[baseline=(current bounding box)]
		\filldraw  (0,0) circle (2pt); 
		\filldraw   (0,.5) circle (2pt);
        		\draw (0,0) -- (0,.5); 
        \node[] at (0,.75)  { $k-1$};
		\node[] at (0,-.3)  { $k$};
\end{tikzpicture}, \hspace{1cm} J_{p'}=\begin{tikzpicture}[baseline=(current bounding box)]
		\filldraw  (0,0) circle (2pt); 
		\filldraw   (0,.5) circle (2pt);
        		\draw (0,0) -- (0,.5); 
        \node[] at (0,.8)  { $k$};
		\node[] at (0,-.25)  { $k-1$};
\end{tikzpicture},$$ then $p \sim p'$.
\end{lemma}
\begin{proof}
     We will use the idea from the bijection of West \cite{west-thesis} to give a bijection between $\Av_n(p)$ and $\Av_n(p')$. Let $w\in \Av_n(p)$. 
    \begin{enumerate}
        \item[Step 1:] Leave the entries of $w$ with $p$-rank $k-2$ or less untouched and in their position.
        \item[Step 2:] Let $W$ be the set of positions of $w$ in which an entry of $p$-rank $k-1$ is located, and let $X$ be the set of entries of $w$ that are of $p$-rank $k-1$. Now fill the positions of $W$ with the entries in $X$ from left to right, so that each position $i \in W$ is filled with the smallest entry $x$ of $X$ that has not been placed yet and that is larger than the closest entry of rank $k-2$ on the left of position $i$.
    \end{enumerate}  Let $f(w)$ be the obtained permutation. Note that $f(w)$ avoids $p'$ since the existence of such a pattern in $f (w)$ would mean that the last two entries of that pattern were not placed according to the rule specified above. To prove that this is a bijection, we give the inverse of $f$.
    Let $w' \in \Av_n(p')$.

    \begin{enumerate}
        \item[Step 1:] Leave the entries of $w'$ of $p$-rank $k-2$ or less in their original position 
        \item[Step 2:] Let $W,X$ be defined as above. Remove the elements in $X$ from their positions in $W$. In each step, move the largest element of $X$ that has not yet been placed in the leftmost available slot $j$, and move each element of $X$ that has been weakly to the right of j one notch to the right. Then in each step, each slot $i$ contains an entry that is larger than what was in $i$ before, or an entry that was to the left of $i$ before. Both of these steps result in the new entry in position $i$ having rank $k-1$.
    \end{enumerate}

We omit the proof that these are inverses of each other because the argument is same as the one given in West's thesis \cite{west-thesis}.
\end{proof}

We will use the bijection in Lemma \ref{lem: from-west} to define a bijection to prove Theorem \ref{thm: main2-wilf-3size}.

    Define $q:=\st(p\backslash \{k-1\})$ and $q':=\st(p'\backslash \{k-1\})$.
    Let $w\in \Av_n(p)$. Let $W_1$ be the set of positions of $w$ where the entries of $q$-rank $k-1,k$ are located. Then $W_1$ is a union of maximal integer intervals.
    Similarly, for $w\in \Av_n(p')$, let $W_2$ be the set of positions of $w$ where the entries of $q'$-rank $k-1,k$ are located.
    We will now give a bijection $F$ from the set $\Av_n(p)$ to the set $\Av_n(p')$. Start with a permutation $w\in \Av_n(p)$.
    \begin{enumerate} 
        \item Let $i_1,\dots,i_k$ be the positions of $q$- rank $k$. Multiply on the left by $s_{i_1-1}\dots s_{i_k-1}$.
        \item Applying $f$ from Lemma \ref{lem: from-west} on the resulting permutation.
        \item Multiply on the left by $s_{i_1-1}\dots s_{i_k-1}$.
    \end{enumerate}
    The resulting permutation is our required permutation.

    To finish the proof of Theorem \ref{thm: main2-wilf-3size}, we need to prove that the function $F$ defined above is bijection from $\Av_n(p)$ to $\Av_n(p')$.     
\begin{proof}[Proof of Theorem \ref{thm: main2-wilf-3size}]
    We show that the inverse exists. The inverse $G$ is as follows: Start from a permutation in $\Av_n(p')$.
    \begin{enumerate} 
        \item Let $i_1,\dots,i_k$ be the positions of $q'$- rank $k$. Multiply on the left by $s_{i_1-1}\dots s_{i_k-1}$.
        \item Applying $f^{-1}$ (as it is a bijection) from Lemma \ref{lem: from-west} on the resulting permutation.
        \item Multiply on the left by $s_{i_1-1}\dots s_{i_k-1}$.
    \end{enumerate}
    
Clearly, given a permutation in $\Av_n(p)$, the entries of $q$-rank $k$ are well-defined. Therefore, the map itself is well-defined. We need to show that the image lies in $\Av_n(p')$. Each entry $x$ of $q$-rank $k$ is part of an occurrence of $q$ which involves $x$ and the entry immediately to its left. If not, it would be an occurrence of $p$. This also forces each entry immediately to the left of an entry of $q$-rank $k$ to be of $q$-rank $k-1$. This implies that two entries of $q$-rank $k$ are not adjacent. Hence, applying the transpositions in Step $1$ should give us a permutation that avoids $q$. Now, applying $f$ in Step $2$ takes it to a permutation that is in $\Av_n(q')$. Finally, applying the transposition again will create a permutation in $\Av_n(p')$, by the same argument. The elements of $q$-rank $k$ are formed only by these transpositions and therefore, the map $G$ defined earlier is the inverse of map $F$. Thus, we have exhibited a bijection from $\Av(p)$ to $\Av(p')$.
\end{proof}

\section{Proofs of Theorems \ref{thm: disjoint-sum-shape-w} and \ref{thm: disjoint-sum-wilf}}\label{sec: proofs-3-4}

We begin with the proof of Theorem \ref{thm: disjoint-sum-shape-w}.
\begin{proof}[Proof of Theorem \ref{thm: disjoint-sum-shape-w}]
We will prove the Wilf-equivalences $p+q\sim p+q'$ and 
$p+q \sim p'+q$. Putting them together, we will get the 
Wilf-equivalences:  $$p+q\sim p+q'\sim p'+q'.$$
We only show the first equivalence and omit the proof of the second as it is similar to the first. We will highlight the differences at the end of the proof.
Our proof is bijective.  Let $F$ be a square Ferrers board. 
Fix a transversal $f$ that avoids $p+q$. Notice that if $f$ avoids $p$, then $f$ avoids occurrences of both $p+q$ and $p+q'$.
If $f$ avoids $p$, then map it to itself.
    Otherwise, perform the following. We start with the squares without any colour.
    \begin{enumerate}
    \item Fix a bijection $\Lambda$ from the transversals avoiding $q$ to $q'$.
        \item For any square $(i,j)\in F$, if there is an occurrence of $p$ in the subboard of $F$ that is to the left of $(i,j)$, then we color it gray. Otherwise, we colour it white.
        \item Find the $1$'s coloured white and colour their corresponding rows white.
        \item Remove all the white squares from $F$, denoting this new board, obtained by left and bottom justifying the squares, by $F'$. $F'$ will be square as we start with a square Ferrers board. The induced transversal on $F'$ is denoted by $f'$. 
        \item Replace this transversal $f'$ by $f'':=\Lambda(f')$ and restore all the white squares from the places they were removed.
        \item The transversal that we have created avoids occurrences of $p+q'$.
        \item Map $f$ to this transversal. This gives the required map.
    \end{enumerate}
    For the second equivalence, in our proof, we just have to look for occurrences of $q'$ to the right of cells in Step $2$. The rest of the proof is identical and is omitted.
\end{proof}


Let $p,q$ be two POPs of size $k$ and $l$ respectively with $k+l=n$.  Thus, the size of $p+q$ is $n$. From the result of Gao and Kitaev \cite[Theorem 1]{gao-kitaev-length-4-5-pops}, the reverse $r(p)$ of any POP $p$ of size $n$ can be obtained by performing the complementation operation on its labels, that is, replacing a label $i$ by $n+1-i$. Therefore, $r(p+q)$ is the reverse of the POP $p+q$.

Let $P,Q$ be the underlying sets of the POPs $p,q$ respectively. Let $r_P(p+q)$ be the operation on $P\bigsqcup Q$ that replaces a label $i$ of $P$ by $k+1-i$ while keeping the labels on the 
set $Q$ fixed. Similarly, $r_Q(p+q)$ is the operation on $P\bigsqcup Q$ that replaces a label $i'$ of $Q$ by $l+1-i'$ while keeping the labels on the
set $P$ fixed. Our next result explores the effect of a composition of these reversal operations.
\begin{remark}
    Informally, the underlying set of the poset $p+q$ is $P\bigsqcup Q$. $r_P$ changes only the labels of $P$ in $p+q$ and $r_Q$ changes only the labels of $Q$ in $p+q$. We do not change the Hasse diagram of the underlying poset.
\end{remark}
\begin{lemma}\label{p+q to q+p}
      $r_Q(r(r_P(p+q))=q+p$.
\end{lemma}
\begin{proof}
    Before starting the proof, note that any label $i$ and $i'$ of $P$ and $Q$ in $p+q$ become 
    $i+l$ and $i'-k$ in $q+p$ respectively. We prove the lemma in the following way. Let $i, j$ 
be two labels of elements $u,v$ in $P$ such that $u<_p v$ in $p+q$. Consider $r_P(p+q)$ where 
the labels $i,j$ become $k-i,k-j$. Now, consider $r(r_P(p+q))$. The labels become $n-k+i,n-k+j$ 
which is nothing but $l+i,l+j$. Lastly, if we consider $r_Q$, then it will not affect the labels 
of $P$. Hence, $i,j$ becomes $l+i,l+j$ in $r_Q(r(r_P(p+q))$. Therefore, the labels of 
$u,v \in P$ change from $i,j$ to $l+i,l+j$ 
respectively.  

 Now, we will follow the same steps for the labels of $q$. Consider two labels $i',j'$ from
 elements $u,v$ in $Q$ such that $u<_q v$ in $p+q$. If we consider $r_P$, it will not affect the 
 labels of elements in $Q$. Now, do the reverse $r$ of $p+q$. We have $n+1-i', n+1-j'$ for $u,v$ 
 in $r(r_P(p+q))$. Lastly, consider the operation $r_Q$. We have the labels $i'-k, j'-k$ for the 
 elements $u,v$ in $r_Q(r(r_P(p+q))$. Therefore, the labels of $u,v \in Q$ change 
 from $i',j'$ to $i'-k,j'-k$ respectively.  As this is the labeling of the labeled poset $q+p$, thiscompletes our proof.
\end{proof}
We illustrate Lemma \ref{p+q to q+p}.
Let $p=$ \begin{tikzpicture}[baseline=(current bounding box)]
		\filldraw  (0,.25) circle (2pt);
        \filldraw  (-.25,-.25) circle (2pt);
        \filldraw  (.25,-.25) circle (2pt);
		
		\draw (0,.25) -- (-.25,-.25); 
        \draw (0,0.25) -- (0.25,-.25); 
        \node[] at (0,.5)   { $1$};
		\node[] at (-.5,-.25)  { $2$};
		\node[] at (.5,-.25) { $3$};
        
\end{tikzpicture} and $q=$ \begin{tikzpicture}[baseline=(current bounding box)]
    \filldraw  (0,0.25) circle (2pt);
    \filldraw  (0,-.25) circle (2pt);
    \draw (0,.25) -- (0,-.25); 
    \node[] at (0.3,.25)   { $1$};
    \node[] at (0.3,-.25)   { $2$};
\end{tikzpicture}. Hence, $p+q=$ \begin{tikzpicture}[baseline=(current bounding box)]
		\filldraw  (0,.25) circle (2pt);
        \filldraw  (-.25,-.25) circle (2pt);
        \filldraw  (.25,-.25) circle (2pt);
        \filldraw  (0.8,0.25) circle (2pt);
        \filldraw  (.8,-.25) circle (2pt);
		
		\draw (0,.25) -- (-.25,-.25); 
        \draw (0,0.25) -- (0.25,-.25); 
        \draw (.8,0.25) -- (0.8,-.25); 
        \node[] at (0,.5)   { $1$};
		\node[] at (-.5,-.25)  { $2$};
		\node[] at (.5,-.25) { $3$};
        \node[] at (1,.25)   { $4$};
        \node[] at (1,-.25)   { $5$};
\end{tikzpicture}.

\begin{tikzpicture}[baseline=(current bounding box)]
		\filldraw  (0,.25) circle (2pt);
        \filldraw  (-.25,-.25) circle (2pt);
        \filldraw  (.25,-.25) circle (2pt);
        \filldraw  (.8,0.25) circle (2pt);
        \filldraw  (0.8,-.25) circle (2pt);
		
		\draw (0,.25) -- (-.25,-.25); 
        \draw (0,0.25) -- (0.25,-.25); 
        \draw (0.8,0.25) -- (.8,-.25); 
        \node[] at (0,.5)   { $1$};
		\node[] at (-.5,-.25)  { $2$};
		\node[] at (.5,-.25) { $3$};
        \node[] at (1,.25)   { $4$};
        \node[] at (1,-.25)   { $5$};
        
\end{tikzpicture} $\xlongrightarrow{r_p} $ \begin{tikzpicture}[baseline=(current bounding box)]
		\filldraw  (0,.25) circle (2pt);
        \filldraw  (-.25,-.25) circle (2pt);
        \filldraw  (.25,-.25) circle (2pt);
        \filldraw  (.8,0.25) circle (2pt);
        \filldraw  (0.8,-.25) circle (2pt);
		
		\draw (0,.25) -- (-.25,-.25); 
        \draw (0,0.25) -- (0.25,-.25); 
        \draw (0.8,0.25) -- (.8,-.25); 
        \node[] at (0,.5)   { $3$};
		\node[] at (-.5,-.25)  { $2$};
		\node[] at (.5,-.25) { $1$};
        \node[] at (1,.25)   { $4$};
        \node[] at (1,-.25)   { $5$};
\end{tikzpicture} $\xlongrightarrow{r} $  \begin{tikzpicture}[baseline=(current bounding box)]
		\filldraw  (0,.25) circle (2pt);
        \filldraw  (-.25,-.25) circle (2pt);
        \filldraw  (.25,-.25) circle (2pt);
        \filldraw  (.8,0.25) circle (2pt);
        \filldraw  (0.8,-.25) circle (2pt);
		
		\draw (0,.25) -- (-.25,-.25); 
        \draw (0,0.25) -- (0.25,-.25); 
        \draw (0.8,0.25) -- (.8,-.25); 
        \node[] at (0,.5)   { $3$};
		\node[] at (-.5,-.25)  { $4$};
		\node[] at (.5,-.25) { $5$};
        \node[] at (1,.25)   { $2$};
        \node[] at (1,-.25)   { $1$};
\end{tikzpicture} $\xlongrightarrow{r_q} $  \begin{tikzpicture}[baseline=(current bounding box)]
		\filldraw  (0,.25) circle (2pt);
        \filldraw  (-.25,-.25) circle (2pt);
        \filldraw  (.25,-.25) circle (2pt);
        \filldraw  (-.8,0.25) circle (2pt);
        \filldraw  (-.8,-.25) circle (2pt);
		
		\draw (0,.25) -- (-.25,-.25); 
        \draw (0,0.25) -- (0.25,-.25); 
        \draw (-.8,0.25) -- (-.8,-.25); 
        \node[] at (0,0.5)   { $3$};
		\node[] at (-.8,-.5)  { $2$};
		\node[] at (.4,-.5) { $5$};
        \node[] at (-.8,.5)   { $1$};
        \node[] at (-.35,-.5)   { $4$};
        \node[] at (1.25,0)   { $=q+p.$};
\end{tikzpicture}
We are now ready to prove Theorem \ref{thm: disjoint-sum-wilf}.

\begin{proof}[Proof of Theorem \ref{thm: disjoint-sum-wilf}]
    Let $rev(\pi)$ be the reverse of the permutation $\pi$. 
    If a permutation $\pi$ avoids a POP $p$, then $rev(\pi) $ will avoid $r(p)$. Hence, $p \sim r(p)$. By Theorem \ref{thm: disjoint-sum-shape-w}, we have the following. 
    $$p+q \sim r(p)+q =r_P(p+q) \sim r(r_P(p+q))$$ 
Now, let $r|_P, r|_Q$ be the effect of $r$ on the labels on elements of $P,Q$ respectively. We have
    $$r(r_P(p+q)) = r|_P(r_P(p))+ r|_Q(r_P(q)) \sim r|_P(r_P(p))+ r_Q(r|_Q(r_P(q)).$$ 
    Since $r_Q$ does not affect the labels on $P$, we have
    $$r|_P(r_P(p))+ r_Q(r|_Q(r_P(q))=r_Q(r(r_P(p))+ r_Q(r(r_P(q))=r_Q(r(r_P(p+q)).$$ By Lemma \ref{p+q to q+p}, we have $r_Q(r(r_P(p+q))=q+p$. Hence, $p+q \sim q+p$, completing the proof.
\end{proof}


\section{Classification of POPs of size $3$ whose components are chains }\label{sec: len3}

Before we start the classification of POPs of size $3$, $4$ and $5$ all of whose components are chains, we recall the following result of Gao and Kitaev  \cite[Theorem 4]{gao-kitaev-length-4-5-pops} 
that describes the effect of having isolated vertices with labels that are completely smaller or larger than the labels of the remaining vertices.

\begin{theorem}[Gao and Kitaev]
\label{iso_ver_beginning_end}
    Let \( p \) be a POP of size \( k \geq 1 \), and suppose the set of labels of isolated vertices includes
\[
I = \{1, 2, \ldots, i\} \cup \{k - s + i + 1, k - s + i + 2, \ldots, k\}
\]
for integers \( 0 \leq i \leq s \leq k \). Let \( q \) be the POP obtained from \( p \) by removing the elements corresponding to the labels in \( I \). Then,
\[
|\Av_n(p)| = 
\begin{cases} 
\displaystyle \frac{n!}{(n - s)!} \, |\Av_{n - s}(q)| & \text{if } n < k, \\[10pt]
0 & \text{if } n \geq k.
\end{cases}
\]
\end{theorem}

We move on to the proof of Theorem \ref{thm: len3-B12-B21}. We give an encoding scheme that maps a tranversal avoiding $p$ (similarly, $p'$) of a Ferrers board $F$ to certain words on the alphabet $\{0,1,2\}$. We will use this encoding scheme to give a bijection between $\Av(p)$ and $\Av(p')$ for all Ferrers boards.

Start with an empty Ferrers board whose cells are white. We insert $1$'s in each row of the Ferrers board $F$ going from top to bottom as follows. Suppose we have placed a $1$ in the top row, we color the row and column containing the $1$ gray. Once $k$ topmost rows are gray meaning $1$'s have been placed in these rows, we place the next $1$ in the topmost row with white cells. Subsequently, the row and column where it was placed will be colored gray.

When we insert in this manner to obtain a transversal that avoids $p$ or $p'$, there are some constraints. Call a position/cell \emph{suitable} if placing $1$ in the position does not form an occurrence of $p$ or $p'$ with past or future placements.

\begin{lemma}\label{lem: atmost2}
At any stage of the process, in the topmost row of white cells, there are at most two suitable cells to place $1$ avoiding $p$ or $p'$. For $p$, it is the first and second cells on the top row. For $p'$, it is the first and last cells on the top row.
\end{lemma}
\begin{proof} At the current stage, if the topmost row has size more than $3$, to avoid $p$, we cannot place the $1$ outside of the first two columns. If we did, the $1$'s that would be placed in the first two columns in the future, along with this $1$ would form an occurrence of $p$.
Similarly, to avoid $p'$, we cannot place $1$ outside the first or last column. 
\end{proof}

Next, we examine the case when there is at most $1$ suitable position available in the top row. An instance is when there is exactly one white cell in the top row. Another one is described in the following lemma. 
\begin{lemma}\label{lem: rect}
    Let $\ell(c)$ denote the size of the column $c$. At some stage of the process for $p$ or $p'$, let $F'$ be the board made up of the remaining white cells. Suppose that the top row of $F'$ comes from the columns $c_1<\dots<c_l$ of $F$. If the rectangle with corners $(1,1), (c_l,1), (c_l, \ell(c_l))$, $(1,\ell(c_l))$ (corners included) already contains a $1$ to the left of $c_1$, then there is at most one suitable position. To avoid $p$, it is the leftmost position on the top row. To avoid $p'$, it is the rightmost position on the top row.
\end{lemma}
\begin{proof}
    This follows from the previous lemma, as the $1$ to the left of $c_1$ in the rectangle will form an occurrence of $p$ or $p'$ if placed in the other suitable position.
\end{proof}
\begin{remark}
    Outside of the two cases described above, there are none other. If there is no $1$ in the rectangle described above and there are more than $2$ white cells in the top row, there are always two suitable positions. Otherwise, there had to be a $1$ in the rectangle described in Lemma \ref{lem: rect}.
\end{remark}
\begin{lemma}\label{lem: atleast1}
    At any stage of the process, in the topmost row of white cells, there is at least one suitable position to place $1$ avoiding $p$ or $p'$.
\end{lemma}
\begin{proof}
    To avoid $p$, we can always place $1$ in the leftmost position on the top row of the white cells. To avoid $p'$, we can always place $1$ in the rightmost position on the top row of the white cells. This position will be suitable.
\end{proof}

Using the above scheme, we create an encoding that maps tranversals that avoid $p$ to those that avoid $p'$. The encoding will be a word with the letters $\{0,1,2\}$, obtained in the following way. Let $T$ be a transversal of $F$ that avoids $p$. 
\begin{enumerate}
     \item Let $w=w_1\dots w_n$ be the encoding word.
    \item Remove all the $1$'s and place them back on the board according the process described above.
    \item While placing $1$ on the row $i$ from the top, if there were two suitable positions on the top row and $1$ appears in the first suitable position from the left in $T$, let $w_i=1$. If it appears in the second suitable position from the left in $T$, let $w_i=2$.
    \item While placing $1$ in row $i$, if there were only one suitable position on the top row, then let $w_i=0$.
    \item Once we have placed all the $1$'s, return $w=w_1\dots w_n$.
\end{enumerate}

Now, to obtain a transversal of $F$ that avoids $p'$, we invert the encoding as follows.
\begin{enumerate}
    \item Let $w=w_1\dots w_n$ be the encoding word.
    \item If $w_i=1$, place $1$ in the first suitable position from the left, in the topmost row of white cells. If $w_i=2$, place $1$ in the second suitable position, from the left, in the topmost row of white cells.
    \item If $w_i=0$, place it in the only suitable position available in the top row of white cells.
    \item Return the transversal once all $1$'s are placed.
\end{enumerate}

Similarly, we may start with a transversal that avoids $p'$, obtain an encoding word and then, invert it to get a transversal that avoids $p$. We will prove that these are inverses.

In the insertion process for $p$ or for $p'$, the shape of the board of white cells after $i$ iterations will be the same (up to squashing and left aligning). Therefore, if there were only one white cell on the top row for $p$, it would be the case for $p'$ and vice versa. We will often talk about the Ferrers board of white cells. Even though, they are not technically Ferrers boards, we use it to mean the Ferrers board obtained by squashing and left aligning the white cells. We will call an insertion stage \emph{forced} if there is only one suitable cell to insert at.
\begin{lemma}\label{lem: forced-descent}
    At some stage $i$, if $w_i=1$ and the size of the top most row of white cells is strictly greater than $2$, then the subsequent placements are forced until a stage is encountered where the size of the top row of white cells is exactly $2$. The stage after that is not forced unless there is only one white cell in the top row. This happens at the same stage for $p$ and $p'$.
\end{lemma}
\begin{proof}
    We prove the lemma first for $p$. Let $F'$ be the Ferrers board of white cells at stage $i$. If the top row of $F'$ has $3$ or more cells, then so do subsequent rows. Now, if we place $1$ in the first column in the first row, then we would have to place $1$ in the second column in the second row. If not, when we place a $1$ in the second column of $F'$ in the future, we would have an occurrence of $p$. Arguing similarly, every placement is forced until we encounter a stage where the size of the top row of white cells is $2$. Let this be stage $j$.

   At stage $j-1$, let $\hat{F'}$ be the board of white cells. By definition, the top two rows have $3$ cells. At stage $j-1$, we would have placed in the first column of $\hat{F'}$ as this was forced. At stage $j$, we would have to place in the second column of $\hat{F'}$. This is also forced. If the third row of $\hat{F'}$ has $3$ cells, then there would only be one cell for stage $j+1$. Else, if there are more than $3$ cells in the third row of $\hat{F'}$, we can place $1$ in the fourth column on the third row and it would not participate any occurrence of $p$ with pre-existing $1'$s as the previous rows do not have cells in this column. This proves the lemma for $p$.

    Now, for $p'$, let $F''$ be Ferrers board of white cells. If we place $1$ in the first column, then in the second row, we cannot place $1$ in the second column, by a similar reasoning as before. Therefore, we are forced to place in the last cell on the second row.
    Arguing similarly, every placement is forced until we reach a row where only the second and third columns of $F''$ have white cells and the rest of the row is gray. Let stage $j$ be the stage when this happens.

    At stage $j-1$, let $\hat{F''}$ be the board of white cells. We have three cells in the top two rows, which correspond to the second, third and fourth columns of $F''$. In stage $j-1$, we have to place $1$ in the rightmost cell (fourth column). In stage $j$, once again, we have to place $1$ in the rightmost cell of that row (third column). If the third row of $\hat{F''}$ had $3$ cells to begin with, then, in stage $j+1$, there would be only one cell to place in. Else, we can place $1$ in the cell that is in the second column of $F''$ or in the rightmost column on the third row of $\hat{F''}$.
    By way of contradiction, suppose this stage were forced. Then, there would need to be a $1$ to the left of the first column in the rectangle described in Lemma \ref{lem: rect}.
    However, the rectangle described in Lemma \ref{lem: rect} would not have any $1$'s to the left of the first column of $\hat{F''}$ as the rectangle lies entirely below the top $j$ rows. This is because the last column of the third row of $\hat{F''}$ has no cells in the higher rows of $F''$. If there were cells in higher rows, there would also be a cell in the second row of $\hat{F''}$ in this column but the second row had strictly smaller size than the third row. This proves that this stage is not forced.
\end{proof}

\begin{lemma}\label{lem: unforced-descent}
    If $w_i=2$, then the stage $i+1$ is not forced for $p$ or $p'$ unless the top row has only one white cell.
\end{lemma}
\begin{proof}
    We prove the lemma first for $p$. Let $F'$ be the board of white cells at stage $i$. The top row of $F'$ has at least $2$ cells. If the first and second row both have only $2$ cells, then the top row at stage $i+1$ will have only one cell. For the general case, suppose there were $1$ to the left of the first column in the rectangle described in Lemma \ref{lem: rect} for stage $i+1$, then this $1$ would also be in the rectangle for insertion at stage $i$, thereby forcing $w_i=0$, which is a contradiction.

    Now, for $p'$, let $F''$ be the board of white cells at stage $i$. The analysis for the special case is similar. For the general case, suppose there were a $1$ to the left of the first column in the rectangle described in Lemma \ref{lem: rect} for stage $i+1$, then this would also be to the left of the first column in the rectangle for the insertion at stage $i$. This would mean that insertion at stage $i$ was forced, which is a contradiction.
\end{proof}

\begin{remark}
    When we reach a stage $i$ where there is only one white cell in the top row, it means that, at stage $i+1$, the first $i$ columns and top $i$ rows of the board are gray. The remaining white board is completely to the right and below the already placed $1$'s and no $1$ placed in the current white board, in the future, can form an occurrence of $p$ or $p'$ with a $1$ in the first $i$ columns. This is because the $(i+1)$-th column has no cells in the $i$-th row from the top.
\end{remark}

With the help of these lemmas, we are ready to prove that this is a bijection.

\begin{proof}[Proof of Theorem \ref{thm: len3-B12-B21}]
    We claim that the encoding scheme gives a bijection. We proceed by induction on the size of the Ferrers board. It is easily verified for Ferrers boards of size $3$, for example. For smaller boards, all tranversal avoid $p$ and $p'$. Now, suppose it is true for all Ferrers board of size $<k$. We prove for boards of size $k$. Let $F$ be a Ferrers board of size $k$. If the top most row of $F$ has one cell, then we argue as follows. Let $F'$ be the board obtained by deleting the first column of $F$. Any tranversal of $F'$ avoiding $p$ or $p'$ can be extended to a transversal of $F$ avoiding $p$ or $p'$ and each tranversal can be restricted to a transversal of $F'$. Therefore, the result is true by induction.

    Suppose the first row has more than $1$ cell, then Lemmas \ref{lem: forced-descent}, \ref{lem: unforced-descent} ensure that we reach a position where we have filled the first $i$ columns and top $i$ rows for some $i$ (see the Remark preceding the proof). Here, similar to our previous argument, we have that the white cells form a smaller Ferrers board $F'$ and we can use the bijection on $F'$ to construct the bijection on $F$ through induction.
\end{proof}

\begin{example}
We work out an example for the board $(5,5,4,4,3)$. We can see that the transversal $4,5,3,1,2$ avoids $p$. We obtain the encoding word from the board and transversal. \vspace{0.2cm}\\
\scalebox{0.75}{
\ytableausetup{centertableaux} 
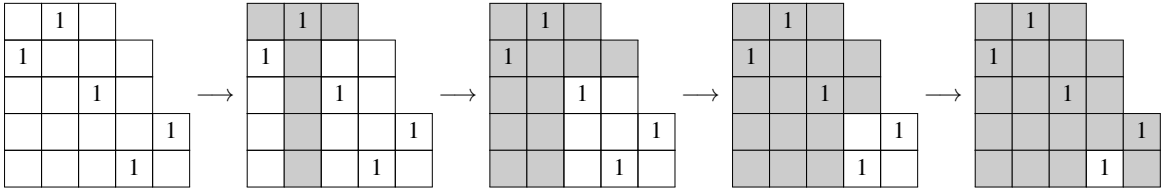
\begin{ytableau}
   \textcolor{white}{d}  & $1$ & \\
     $1$ & \textcolor{white}{d} & \textcolor{white}{d} & \\
    \textcolor{white}{d} & & $1$ & \textcolor{white}{d}  \\
     & \textcolor{white}{d} & \textcolor{white}{d} & & $1$\\
     & \textcolor{white}{d} & \textcolor{white}{d} & $1$ & \\
\end{ytableau}
$\longrightarrow$
\ytableausetup{centertableaux} 
\begin{ytableau}
   *(lgrey)  & *(lgrey)$1$ & *(lgrey)\\
     $1$ & *(lgrey) & \textcolor{white}{d} & \\
    \textcolor{white}{d} & *(lgrey) & $1$ & \textcolor{white}{d} \\
     & *(lgrey) & \textcolor{white}{d} & & $1$\\
     & *(lgrey) & \textcolor{white}{d} & $1$ & \\
\end{ytableau}
$\longrightarrow$
\ytableausetup{centertableaux} 
\begin{ytableau}
   *(lgrey)  & *(lgrey)$1$ & *(lgrey)\\
     *(lgrey)$1$ & *(lgrey) & *(lgrey) & *(lgrey)\\
    *(lgrey) & *(lgrey) & $1$ & \textcolor{white}{d} \\
     *(lgrey)& *(lgrey) &\textcolor{white}{d} & & $1$\\
     *(lgrey)& *(lgrey) & \textcolor{white}{d} & $1$ & \\
\end{ytableau}
$\longrightarrow$
\ytableausetup{centertableaux} 
\begin{ytableau}
   *(lgrey)  & *(lgrey)$1$ & *(lgrey)\\
     *(lgrey)$1$ & *(lgrey) & *(lgrey) & *(lgrey)\\
    *(lgrey) & *(lgrey) & *(lgrey)$1$ & *(lgrey) \\
     *(lgrey)& *(lgrey) & *(lgrey) & & $1$\\
     *(lgrey)& *(lgrey) & *(lgrey) & $1$ & \\
\end{ytableau}
$\longrightarrow$
\ytableausetup{centertableaux} 
\begin{ytableau}
   *(lgrey)  & *(lgrey)$1$ & *(lgrey)\\
     *(lgrey)$1$ & *(lgrey) & *(lgrey) & *(lgrey)\\
    *(lgrey) & *(lgrey) & *(lgrey)$1$ & *(lgrey)  \\
     *(lgrey)& *(lgrey) & *(lgrey) & *(lgrey)& *(lgrey)$1$\\
     *(lgrey)& *(lgrey) & *(lgrey) & $1$ & *(lgrey)\\
\end{ytableau}
}\vspace{0.2cm}

The encoding word is $2,1,0,2,0$. The boards below display the insertion process for the transversal avoiding $p'$.\vspace{0.2cm}\\
\scalebox{0.75}{
\ytableausetup{centertableaux} 
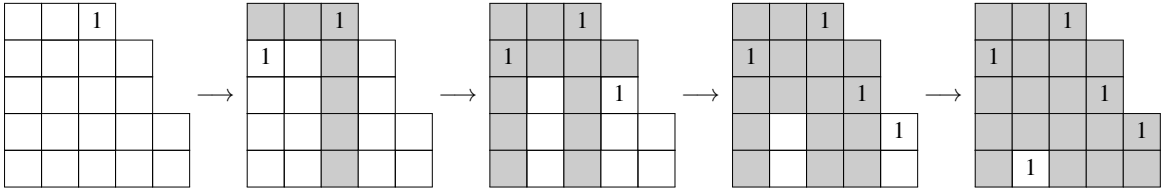
\begin{ytableau}
   *(white)  &  & $1$\\
     *(white) & \textcolor{white}{d} & *(white) & \\
    \textcolor{white}{d} & & *(white) & \textcolor{white}{d} \\
     & *(white) & \textcolor{white}{d} & & *(white)\\
     & \textcolor{white}{d} & \textcolor{white}{d} & *(white) & \\
\end{ytableau}
$\longrightarrow$
\ytableausetup{centertableaux} 
\begin{ytableau}
   *(lgrey) & *(lgrey) & *(lgrey)$1$\\
     $1$ & \textcolor{white}{d} & *(lgrey) & \\
    \textcolor{white}{d} & & *(lgrey) & \textcolor{white}{d}  \\
     & *(white) & *(lgrey) & & *(white)\\
     & \textcolor{white}{d} & *(lgrey) & *(white) & \\
\end{ytableau}
$\longrightarrow$
\ytableausetup{centertableaux} 
\begin{ytableau}
   *(lgrey) & *(lgrey) & *(lgrey)$1$\\
     *(lgrey)$1$ & *(lgrey) & *(lgrey) & *(lgrey)\\
    *(lgrey) & & *(lgrey) & $1$ \\
     *(lgrey)& *(white) & *(lgrey) & & *(white)\\
     *(lgrey)& \textcolor{white}{d} & *(lgrey) & *(white) & \\
\end{ytableau}
$\longrightarrow$
\ytableausetup{centertableaux} 
\begin{ytableau}
   *(lgrey) & *(lgrey) & *(lgrey)$1$\\
     *(lgrey)$1$ & *(lgrey) & *(lgrey) & *(lgrey)\\
    *(lgrey) & *(lgrey) & *(lgrey) & *(lgrey)$1$ \\
     *(lgrey)& *(white) & *(lgrey) & *(lgrey)& $1$\\
     *(lgrey)& \textcolor{white}{d} & *(lgrey) & *(lgrey) & \\
\end{ytableau}
$\longrightarrow$

\ytableausetup{centertableaux} 
\begin{ytableau}
   *(lgrey) & *(lgrey) & *(lgrey)$1$\\
     *(lgrey)$1$ & *(lgrey) & *(lgrey) & *(lgrey)\\
    *(lgrey) & *(lgrey) & *(lgrey) & *(lgrey)$1$ \\
     *(lgrey)& *(lgrey) & *(lgrey) & *(lgrey)& *(lgrey)$1$\\
     *(lgrey)& $1$ & *(lgrey) & *(lgrey) & *(lgrey)\\
\end{ytableau}
}\vspace{0.2cm}\\

The transversal obtained is $4,1,5,3,2$ which avoids $p'$.
\end{example}

The connected POPs of size $3$ are the classical patterns of size $3$ and have received significant attention in previous studies. The number of permutations avoiding the pattern $123$, $|\Av_n(123)|$, was done by MacMahon \cite{macmahon-book} and the number of permutations avoiding the pattern $132$, $|\Av_n(132)|$, was done by Knuth \cite{knuth}.  Both cardinalities turn out to be the $n$-th Catalan number $C_n$. All other classical patterns of size $3$ can be obtained from these two by applying reverse and complement. Therefore, for any classical pattern $p$ of size $3$ we have $|\Av_n(p)|=C_n=\frac{1}{n+1}\binom{2n}{n}$. The shape-Wilf-equivalence of the classical patterns of size $3$ 
\begin{tikzpicture}[baseline=(current bounding box)]
    \filldraw  (0,0) circle (2pt);
    \filldraw  (0,-.5) circle (2pt);
    \filldraw  (0,0.5) circle (2pt);
    \draw (0,0) -- (0,-.5); 
    \draw (0,0) -- (0,.5);
    \node[] at (0.3,0) {$1$};
    \node[] at (0.3,.5)   { $2$};
    \node[] at (0.3,-.5)   { $3$};
\end{tikzpicture} $\sim_s$ \begin{tikzpicture}[baseline=(current bounding box)]
    \filldraw  (0,0) circle (2pt);
    \filldraw  (0,-.5) circle (2pt);
    \filldraw  (0,0.5) circle (2pt);
    \draw (0,0) -- (0,-.5); 
    \draw (0,0) -- (0,.5);
    \node[] at (0.3,0) {$3$};
    \node[] at (0.3,.5)   { $1$};
    \node[] at (0.3,-.5)   { $2$};
\end{tikzpicture} was proved by Stankova, West \cite{stankova_west} and 
\begin{tikzpicture}[baseline=(current bounding box)]
    \filldraw  (0,0) circle (2pt);
    \filldraw  (0,-.5) circle (2pt);
    \filldraw  (0,0.5) circle (2pt);
    \draw (0,0) -- (0,-.5); 
    \draw (0,0) -- (0,.5);
    \node[] at (0.3,0) {$2$};
    \node[] at (0.3,.5)   { $1$};
    \node[] at (0.3,-.5)   { $3$};
\end{tikzpicture}$\sim_s$ \begin{tikzpicture}[baseline=(current bounding box)]
    \filldraw  (0,0) circle (2pt);
    \filldraw  (0,-.5) circle (2pt);
    \filldraw  (0,0.5) circle (2pt);
    \draw (0,0) -- (0,-.5); 
    \draw (0,0) -- (0,.5);
    \node[] at (0.3,0) {$1$};
    \node[] at (0.3,.5)   { $3$};
    \node[] at (0.3,-.5)   { $2$};
\end{tikzpicture} $\sim_s$ \begin{tikzpicture}[baseline=(current bounding box)]
    \filldraw  (0,0) circle (2pt);
    \filldraw  (0,-.5) circle (2pt);
    \filldraw  (0,0.5) circle (2pt);
    \draw (0,0) -- (0,-.5); 
    \draw (0,0) -- (0,.5);
    \node[] at (0.3,0) {$2$};
    \node[] at (0.3,.5)   { $3$};
    \node[] at (0.3,-.5)   { $1$};
\end{tikzpicture} was proved by Babson, West \cite{babson} and Backelin, West, Xin \cite{BWX_conf}.

Now we consider the POPs of size $3$ with $2$ linear components. Theorem \ref{thm: len3-B12-B21} shows that \begin{tikzpicture}[baseline=(current bounding box)]
    \filldraw  (1,0.25) circle (2pt);
    \filldraw  (1,-.25) circle (2pt);
    \filldraw  (0.5,0) circle (2pt);
    \draw (1,.25) -- (1,-.25); 
    \node[] at (0.3,0) {$1$};
    \node[] at (1,.55)   { $2$};
    \node[] at (1,-.55)   { $3$};
\end{tikzpicture} $\sim_s$ \begin{tikzpicture}[baseline=(current bounding box)]
    \filldraw  (1,0.25) circle (2pt);
    \filldraw  (1,-.25) circle (2pt);
    \filldraw  (0.5,0) circle (2pt);
    \draw (1,.25) -- (1,-.25); 
    \node[] at (0.3,0) {$1$};
    \node[] at (1,.55)   { $3$};
    \node[] at (1,-.55)   { $2$};
\end{tikzpicture}. By Theorem \ref{iso_ver_beginning_end}, we have \begin{tikzpicture}[baseline=(current bounding box)]
    \filldraw  (1,0.25) circle (2pt);
    \filldraw  (1,-.25) circle (2pt);
    \filldraw  (0.5,0) circle (2pt);
    \draw (1,.25) -- (1,-.25); 
    \node[] at (0.3,0) {$1$};
    \node[] at (1,.55)   { $b$};
    \node[] at (1,-.55)   { $c$};
\end{tikzpicture} $\sim$ \begin{tikzpicture}[baseline=(current bounding box)]
    \filldraw  (0,0.25) circle (2pt);
    \filldraw  (0,-.25) circle (2pt);
    \filldraw  (0.5,0) circle (2pt);
    \draw (0,.25) -- (0,-.25); 
    \node[] at (0.7,0) {$3$};
    \node[] at (0,.55)   { $b-1$};
    \node[] at (0,-.55)   { $c-1$};
\end{tikzpicture} for $\{b,c\}=\{2,3\}$, but they are not shape-Wilf-equivalent. 
The POP \begin{tikzpicture}[baseline=(current bounding box)]
    \filldraw  (0,0.25) circle (2pt);
    \filldraw  (0,-.25) circle (2pt);
    \filldraw  (0.5,0) circle (2pt);
    \draw (0,.25) -- (0,-.25); 
    \node[] at (0.7,0) {$2$};
    \node[] at (0,.55)   { $1$};
    \node[] at (0,-.55)   { $3$};
\end{tikzpicture}is the reverse of \begin{tikzpicture}[baseline=(current bounding box)]
    \filldraw  (0,0.25) circle (2pt);
    \filldraw  (0,-.25) circle (2pt);
    \filldraw  (0.5,0) circle (2pt);
    \draw (0,.25) -- (0,-.25); 
    \node[] at (0.7,0) {$2$};
    \node[] at (0,.55)   { $3$};
    \node[] at (0,-.55)   { $1$};
\end{tikzpicture}. So, they are Wilf-equivalent but not shape-Wilf-equivalent. Detailed classification of these POPs is provided below in Table\ref{tab:table-len-3}.

\begin{table}[!htbp]

    \centering
\begin{tabular}{|l|l|l|l|l|l|l|l|l|l|}
\noalign{\hrule height 1pt}
\textbf{No. }&\textbf{POP} $(p)$ & $|\Av_n(p)|~(n=1,2,\ldots,8,\ldots)$ & \textbf{OEIS} \\
\noalign{\hrule height 1pt}
1 & \begin{tikzpicture}[baseline=(current bounding box)]
    \filldraw  (0,0) circle (2pt);
    \filldraw  (0,-.5) circle (2pt);
    \filldraw  (0,0.5) circle (2pt);
    \draw (0,0) -- (0,-.5); 
    \draw (0,0) -- (0,.5);
    \node[] at (0.3,0) {$2$};
    \node[] at (0.3,.5)   { $1$};
    \node[] at (0.3,-.5)   { $3$};
\end{tikzpicture} & $1,2,5,14,42,132,429, 1430, \ldots$ & \href{https://oeis.org/search?q=A000108&language=english&go=Search}{A000108} \\

 \noalign{\hrule height 1.5pt}

2 &  \begin{tikzpicture}[baseline=(current bounding box)]
    \filldraw  (0,0.25) circle (2pt);
    \filldraw  (0,-.25) circle (2pt);
    \filldraw  (0.5,0) circle (2pt);
    \draw (0,.25) -- (0,-.25); 
    \node[] at (0.7,0) {$3$};
    \node[] at (-.25,.3)   { $1$};
    \node[] at (-.25,-.3)   { $2$};
\end{tikzpicture} & $1,2,3,4,5,6,7,8, \ldots$ &\href{https://oeis.org/search?q=A000027&language=english&go=Search}{A000027} \\

 \noalign{\hrule height 1.5pt}

3 & \begin{tikzpicture}[baseline=(current bounding box)]
    \filldraw  (0,0.25) circle (2pt);
    \filldraw  (0,-.25) circle (2pt);
    \filldraw  (0.5,0) circle (2pt);
    \draw (0,.25) -- (0,-.25); 
    \node[] at (0.7,0) {$2$};
    \node[] at (-.25,.3)   { $1$};
    \node[] at (-.25,-.3)   { $3$};
\end{tikzpicture} & $1,2,3,5,8,13,21,34,\ldots$ & \href{https://oeis.org/search?q=A000045&language=english&go=Search}{A000045}\\

\hline
\end{tabular}
\caption{Classification of POPs of size $3$}
\label{tab:table-len-3}
\end{table}


We will apply the Theorems \ref{thm: main1-shape}, \ref{thm: main2-wilf-3size}, \ref{thm: disjoint-sum-shape-w}, \ref{thm: disjoint-sum-wilf}, \ref{iso_ver_beginning_end} to classify patterns of size $4$ and $5$ consisting of chain components. For a given pattern $p$, our computations for $\Av_n(p)$ reveal multiple non-trivial equivalence classes, each represented by a row. Adjacent rows without separating lines indicate Wilf-equivalent classes, as discussed below. Lines separate rows which are not Wilf-equivalent.
Rows separated by a thick line represent patterns with distinct isolated vertices. Whenever the sequence is present on OEIS, we provide the OEIS entry. This convention will be followed throughout the rest of the paper.
\section{Classification of POPs of size $4$  whose components are chains }\label{sec: len4}

In this section, we will classify all POPs of size $4$  all of whose components are chains. There are four types of such POPs, which are

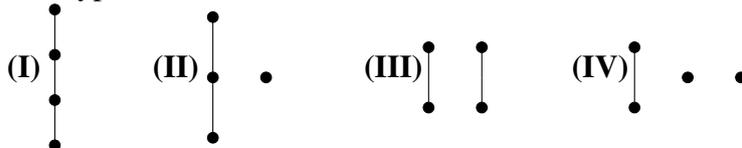
\textbf{(I)} \begin{tikzpicture}[baseline=(current bounding box)]
		\filldraw  (0,.9) circle (2pt);
        \filldraw  (0,.3) circle (2pt);
        \filldraw  (0,-.3) circle (2pt);
        \filldraw  (0,-.9) circle (2pt);
		
		\draw (0,.9) -- (0,.3); 
        \draw (0,0.3) -- (0,-.3); 
        \draw (0,-0.3) -- (0,-.9); 
		
	\end{tikzpicture} \hspace{1cm}
\textbf{(II)} \begin{tikzpicture}[baseline=(current bounding box)]
		\filldraw  (0,0) circle (2pt);
        \filldraw  (0,.8) circle (2pt);
        \filldraw  (0,-.8) circle (2pt);
        \filldraw  (.7,0) circle (2pt);
		
		\draw (0,0) -- (0,.8); 
        \draw (0,0) -- (0,-.8); 
		
	\end{tikzpicture} \hspace{1cm}
\textbf{(III)}\begin{tikzpicture}[baseline=(current bounding box)]
		\filldraw  (0,-.4) circle (2pt); 
		\filldraw   (0,.4) circle (2pt);
		\filldraw  (.7,0.4) circle (2pt);
        \filldraw  (.7,-.4) circle (2pt);
		
		\draw (0,-.4) -- (0,.4); 
        \draw (.7,0) -- (.7,.4); 
        \draw (.7,0) -- (.7,-.4); 
		
	\end{tikzpicture}\hspace{1cm}
\textbf{(IV)}\begin{tikzpicture}[baseline=(current bounding box)]
		\filldraw  (0,-.4) circle (2pt); 
		\filldraw   (0,.4) circle (2pt);
		\filldraw  (.7,0) circle (2pt);
        \filldraw  (1.4,0) circle (2pt);
		
		\draw (0,-.4) -- (0,.4); 
		
	\end{tikzpicture}

Type I are the classical patterns of length $4$ and are classified in \cite{Classical-4-len-pattern}, \cite{noy-elizalde-adv-appl},\cite{mireille},\cite{jaggard-prefix-involution}, \cite{forbidden_seq},\cite{1342} and type IV is classified in \cite{gao-kitaev-length-4-5-pops}. Therefore, we are only left with type II and III. 

\subsection{Type II}\label{len4_type2}
Let the POP \begin{tikzpicture}[baseline=(current bounding box)]\filldraw  (0,0) circle (2pt);
        \filldraw  (0,.8) circle (2pt);
        \filldraw  (0,-.8) circle (2pt);
        \filldraw  (.7,0) circle (2pt);
		
		\draw (0,0) -- (0,.8); 
        \draw (0,0) -- (0,-.8); 
        \node[] at (-.2,.8)  { $a$};
		\node[] at (-.2,0)  { $b$};
		\node[] at (-.2,-.8)  { $c$};
        \node[] at (.9,0)  { $d$};
\end{tikzpicture} be denoted by $(a,b,c;d)$.
 There are $4$ possible values for $d$, and each value corresponds to $3!=6$ different patterns. Hence, we have a total of $24$ many POPs of type II. The complement and reverse of $(a,b,c;d)$ are $(c,b,a;d)$ and $(5-a,5-b,5-c;5-d)$ respectively. Therefore, POPs with $d=1,2$ are reverse of the POPs with $d=4,3$ respectively. 
 As complementation and reversing will give Wilf-equivalent POPs, it is enough to consider the POPs with $d=3,4$. 
 
  As all the classical patterns of size $3$, that is, all the connected linear POPs of size $3$ are Wilf-equivalent, the POPs $(a,b,c;d)$ where $d=1,4$ are Wilf-equivalent, by Theorem \ref{iso_ver_beginning_end}.
 In \cite[Theorem 24]{gao-kitaev-length-4-5-pops}, the POP $p=(4,1,2;3)$ was studied, and we will prove that all POPs $(a,b,c;d)$ with  $d=3$ are Wilf-equivalent to $p$. Indeed, $(4,1,2;3) \sim (4,2,1;3)$
 by Theorem \ref{thm: main1-shape} and  $(4,2,1;3) \sim (2,4,1;3)$  by Theorem \ref{thm: main2-wilf-3size}. Hence, $p=(4,1,2;3) \sim (2,4,1;3)$.  


Therefore, for type II, there are $2$ Wilf-equivalence classes. A detailed classification of type I POPs is provided below Table \ref{table1}.

\begin{table}[!htbp]
    \centering
\begin{tabular}{|l|l|l|l|l|l|l|l|l|l|}
\noalign{\hrule height 1pt}
\textbf{No. }&\textbf{POP of Type II} $(p)$ & $|\Av_n(p)|~(n=1,2,\ldots,8,\ldots)$ &\textbf{OEIS} \\
\noalign{\hrule height 1pt}
 
 &	$(1,2,3;4)$ & & \\
1 & $(1,3,2;4)$ &  $1,2,6,20,70,252,924,3432, \ldots$ & \href{https://oeis.org/search?q=A000984&language=english&go=Search}{A000984}\\
& $(3,1,2;4)$ & &\\
 
 \noalign{\hrule height 1.5pt}
 
 & $(4,1,2;3)$ &  &\\
2 & $(4,2,1;3)$ & $1,2,6,20,71,264,1015,4002,\ldots$ & \href{https://oeis.org/search?q=A049124&language=english&go=Search}{A049124}\\
& $(2,4,1;3)$ & &\\

\hline

\end{tabular}
\caption{Classification of Type II POPs of size $4$ }
    \label{table1}
\end{table}

\subsection{Type III}\label{len4_type3}
Let the POP \begin{tikzpicture}[baseline=(current bounding box)]
		\filldraw  (0,-.4) circle (2pt); 
		\filldraw   (0,.4) circle (2pt);
		\filldraw  (.7,0.4) circle (2pt);
        \filldraw  (.7,-.4) circle (2pt);
        
		\draw (0,-.4) -- (0,.4); 
        \draw (.7,0) -- (.7,.4); 
        \draw (.7,0) -- (.7,-.4); 
		\node[] at (-.2,.4)  { $a$};
		\node[] at (-.2,-.4)  { $b$};
		\node[] at (.9,.4)  { $c$};
        \node[] at (.9,-.4)  { $d$};
\end{tikzpicture} be denoted by $(a,b;c,d)$. The total number of Type II POPs is $4!=24$. Complement and reverse of $(a,b;c,d)$ are $(b,a;d,c)$ and $(5-a,5-b;5-c,5-d)$ respectively.
As $(a,b;c,d)=(c,d;a,b)$, it is enough to consider the POPs with $(a,b)\in \{(1,2),(1,3),(1,4)\}$.

 The POPs $(1,2;4,3), (1,3;4,2)$ and $(1,4;3,2)$ were studied by Gao and Kitaev in  \cite[Theorem 15,16,18]{gao-kitaev-length-4-5-pops} respectively. Each of these POPs is in a separate Wilf-equivalence class. 
As \begin{tikzpicture}[baseline=(current bounding box)]
    \filldraw  (0,0.25) circle (2pt);
    \filldraw  (0,-.25) circle (2pt);
    
    \draw (0,.25) -- (0,-.25); 
    
    \node[] at (-.25,.3)   { $1$};
    \node[] at (-.25,-.3)   { $2$};
\end{tikzpicture} $\sim$  \begin{tikzpicture}[baseline=(current bounding box)]
    \filldraw  (0,0.25) circle (2pt);
    \filldraw  (0,-.25) circle (2pt);
    
    \draw (0,.25) -- (0,-.25); 
    
    \node[] at (-.25,.3)   { $2$};
    \node[] at (-.25,-.3)   { $1$};
\end{tikzpicture}, we have $(1,2;3,4) \sim (1,2;4,3)$ (by Theorem \ref{thm: disjoint-sum-shape-w} ).
The POPs $(1,3;2,4)$  and $(1,4;2,3)$  are in $2$ 
different Wilf-equivalence classes.

Thus, we get $5$ Wilf-equivalent classes. It is easy to check that 
the remaining POPs are compositions of complement and reverse of POPs mentioned so far. Therefore, there are total $5$ Wilf-equivalence 
classes of type III. A detailed classification of type III POPs is 
provided in Table \ref{table2}.

\begin{table}[!htbp]

    \centering
\begin{tabular}{|l|l|l|l|l|l|l|l|l|l|}
\noalign{\hrule height 1pt}
\textbf{No. }&\textbf{POP of Type III} $(p)$ & $|\Av_n(p)|~(n=1,2,\ldots,8,\ldots)$& \textbf{OEIS}  \\
\noalign{\hrule height 1pt}
 
1 &	$(1,2;4,3)$ &  $1,2,6,18, 50, 130, 322, 770, \ldots$ & \href{https://oeis.org/search?q=1%2C+2%2C+6%2C+18%2C+50%2C+130%2C+322%2C+770&language=english&go=Search}{A048495}\\
&$(1,2;3,4)$ & &\\
 \hline
 
2 & $(1,3;4,2)$ & $1,2,6,18,52,152,444, 1296,\ldots$ &\href{https://oeis.org/search?q=1%2C+2%2C+6%2C+18%2C+52%2C+152%2C+444%2C+1296&language=english&go=Search}{A077835}  \\
\hline

3 & $(1,3;2,4)$& $1,2,6,18, 52,147, 413, 1159   \ldots$ &\\
 \hline

4 & $(1,4;3,2)$& $1, 2, 6, 18,50,134,358,962,  \ldots$ &\href{https://oeis.org/search?q=1%2C+2%2C+6%2C+18%2C+50%2C+134%2C+358%2C+962&language=english&go=Search}{A271897} \\

\hline
5 & $(1,4;2,3)$ & $1,2,6, 18, 53, 156, 460, 1357 \ldots$ &\\

\hline

\end{tabular}

\caption{Classification of Type III POPs of size $4$ }
\label{table2}
\end{table}

\section{Classification of POPs of size $5$  whose components are chains }\label{sec: len5}
In this section, we classify all POPs of size $5$ all of whose components are chains. Linear ordered POPs of size $5$ are the classical patterns of size $5$ that have already been classified by  Clisby, Conway, Guttmann and Inoue
\cite{Classical-5-len-pattern}. Therefore, only the following four types of POPs remain to be classified.

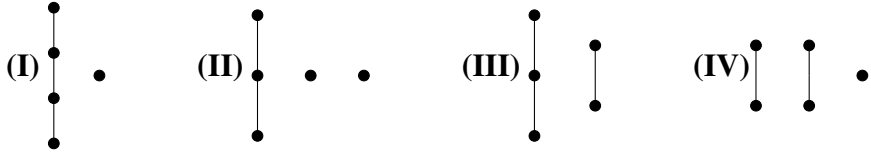
\textbf{(I)} \begin{tikzpicture}[baseline=(current bounding box)]
		\filldraw  (0,.9) circle (2pt);
        \filldraw  (0,.3) circle (2pt);
        \filldraw  (0,-.3) circle (2pt);
        \filldraw  (0,-.9) circle (2pt);
        \filldraw  (.6,0) circle (2pt);
		
		\draw (0,.9) -- (0,.3); 
        \draw (0,0.3) -- (0,-.3); 
        \draw (0,-0.3) -- (0,-.9); 
		
	\end{tikzpicture} \hspace{1cm}
\textbf{(II)} \begin{tikzpicture}[baseline=(current bounding box)]
		\filldraw  (0,0) circle (2pt);
        \filldraw  (0,.8) circle (2pt);
        \filldraw  (0,-.8) circle (2pt);
        \filldraw  (.7,0) circle (2pt);
        \filldraw  (1.4,0) circle (2pt);
		
		\draw (0,0) -- (0,.8); 
        \draw (0,0) -- (0,-.8); 
		
	\end{tikzpicture} \hspace{1cm}
\textbf{(III)} \begin{tikzpicture}[baseline=(current bounding box)]
		\filldraw  (0.8,-.4) circle (2pt); 
		\filldraw   (0.8,.4) circle (2pt);
		\filldraw  (0,0) circle (2pt);
        \filldraw  (0,.8) circle (2pt);
        \filldraw  (0,-.8) circle (2pt);
		
		\draw (0.8,-.4) -- (0.8,.4); 
        \draw (0,0) -- (0,.8); 
        \draw (0,0) -- (0,-.8); 
		
	\end{tikzpicture} \hspace{1cm}
\textbf{(IV)}\begin{tikzpicture}[baseline=(current bounding box)]
		\filldraw  (0,-.4) circle (2pt); 
		\filldraw   (0,.4) circle (2pt);
		\filldraw  (.7,0.4) circle (2pt);
        \filldraw  (.7,-.4) circle (2pt);
        \filldraw  (1.4,0) circle (2pt);
		
		\draw (0,-.4) -- (0,.4); 
        \draw (.7,0) -- (.7,.4); 
        \draw (.7,0) -- (.7,-.4); 
		
	\end{tikzpicture}
\subsection{Type I}
Let the POP \begin{tikzpicture}[baseline=(current bounding box)]
		\filldraw  (0,.9) circle (2pt);
        \filldraw  (0,.3) circle (2pt);
        \filldraw  (0,-.3) circle (2pt);
        \filldraw  (0,-.9) circle (2pt);
        \filldraw  (.6,0) circle (2pt);
		
		\draw (0,.9) -- (0,.3); 
        \draw (0,0.3) -- (0,-.3); 
        \draw (0,-0.3) -- (0,-.9); 
        \node[] at (-.2,0.9)  {$a$};
		\node[] at (-.2,.3)  { $b$};
		\node[] at (-.2,-.3)  {$c$};
        \node[] at (-.2,-.9)  {$d$};
        \node[] at (.8,0)    { $e$};
		
	\end{tikzpicture} be denoted by $(a,b,c,d ;e)$. There are $5$ possible values for $e$, and each value corresponds to $4!=24$ different patterns giving us a total of $120$ POPs of type I. The complement and reverse of $(a,b,c,d;e)$ are $(d,c,b,a;e)$ and $(6-a,6-b,6-c,6-d;6-e)$ respectively.
Therefore, the POPs with $e=1,2$ are the reverse of the POPs with $e=5,4$ respectively. Hence, it is enough to consider the POPs with $e=3,4,5$.

  Classical patterns of size $4$ have been thoroughly studied, and their complete classification has been established. There are the following non-trivial Wilf-equivalences: $$\{1234, 2143, 1243, 1432, 1423\},\{1342, 2413\},\{1324\}.$$
    Other classical patterns of size $4$ are trivially related to the previously mentioned patterns. When $e=5$, by Theorem \ref{iso_ver_beginning_end}, we have $3$ distinct classes of POPs of type I.

   By Theorem \ref{thm: main2-wilf-3size}, we have that $(5,3,2,1;4) \sim (3,5,2,1;4)$ and 
     $(5,3,1,2;4) \sim (3,5,1,2;4)$.
   As \begin{tikzpicture}[baseline=(current bounding box)]
    \filldraw  (0,0) circle (2pt);
    \filldraw  (0,-.5) circle (2pt);
    \filldraw  (0,0.5) circle (2pt);
    \draw (0,0) -- (0,-.5); 
    \draw (0,0) -- (0,.5);
    \node[] at (0.3,0) {$2$};
    \node[] at (0.3,.5)   { $1$};
    \node[] at (0.3,-.5)   { $3$};
\end{tikzpicture}$\sim_s$ \begin{tikzpicture}[baseline=(current bounding box)]
    \filldraw  (0,0) circle (2pt);
    \filldraw  (0,-.5) circle (2pt);
    \filldraw  (0,0.5) circle (2pt);
    \draw (0,0) -- (0,-.5); 
    \draw (0,0) -- (0,.5);
    \node[] at (0.3,0) {$1$};
    \node[] at (0.3,.5)   { $3$};
    \node[] at (0.3,-.5)   { $2$};
\end{tikzpicture} $\sim_s$ \begin{tikzpicture}[baseline=(current bounding box)]
    \filldraw  (0,0) circle (2pt);
    \filldraw  (0,-.5) circle (2pt);
    \filldraw  (0,0.5) circle (2pt);
    \draw (0,0) -- (0,-.5); 
    \draw (0,0) -- (0,.5);
    \node[] at (0.3,0) {$2$};
    \node[] at (0.3,.5)   { $3$};
    \node[] at (0.3,-.5)   { $1$};
\end{tikzpicture},  we have $(5,3,2,1;4) \sim (5,3,1,2;4) \sim (5,1,2,3;4)$ (by Theorem \ref{thm: main1-shape}). Therefore, $(5,3,2,1;4)$, $(3,5,2,1;4)$, $(5,3,1,2;4)$, $(3,5,1,2;4)$ and $(5,1,2,3;4)$ are in the same Wilf-equivalence class. As \begin{tikzpicture}[baseline=(current bounding box)]
		\filldraw  (0,.9) circle (2pt);
        \filldraw  (0,.3) circle (2pt);
        \filldraw  (0,-.3) circle (2pt);
		
		\draw (0,.9) -- (0,.3); 
        \draw (0,0.3) -- (0,-.3); 
        \node[] at (-.2,0.9)  {$2$};
		\node[] at (-.2,.3)  { $1$};
		\node[] at (-.2,-.3)  {$3$};
		
	\end{tikzpicture} $\sim_s$ \begin{tikzpicture}[baseline=(current bounding box)]
		\filldraw  (0,.9) circle (2pt);
        \filldraw  (0,.3) circle (2pt);
        \filldraw  (0,-.3) circle (2pt);
		
		\draw (0,.9) -- (0,.3); 
        \draw (0,0.3) -- (0,-.3); 
        \node[] at (-.2,0.9)  {$1$};
		\node[] at (-.2,.3)  { $3$};
		\node[] at (-.2,-.3)  {$2$};
		
	\end{tikzpicture}, we have  $(5,2,1,3;4)\sim (5,1,3,2;4)$ (by Theorem \ref{thm: main1-shape}).   
As \begin{tikzpicture}[baseline=(current bounding box)]
    \filldraw  (0,0.25) circle (2pt);
    \filldraw  (0,-.25) circle (2pt);
    
    \draw (0,.25) -- (0,-.25); 
    
    \node[] at (-.25,.3)   { $2$};
    \node[] at (-.25,-.3)   { $1$};
\end{tikzpicture} $\sim_s$ \begin{tikzpicture}[baseline=(current bounding box)]
    \filldraw  (0,0.25) circle (2pt);
    \filldraw  (0,-.25) circle (2pt);
    
    \draw (0,.25) -- (0,-.25); 
    
    \node[] at (-.25,.3)   { $1$};
    \node[] at (-.25,-.3)   { $2$};
\end{tikzpicture}, we have $(5,4,2,1;3)\sim (5,4,1,2;3)$ and $ (4,5,2,1;3)\sim (4,5,1,2;3)\sim (5,4,1,2;3)$ (by Theorem \ref{thm: main1-shape}).   
The POPs $(5,2,3,1;4), (2,5,3,1;4),  (1,5,3,2;4),  (2,5,1,3;4)$, $ (1,5,2,3;4)$, $(5,2,4,1;3)$, $(4,2,5,1;3), (2,5,4,1;3), (2,4,5,1;3), (2,5,1,4;3)$ are in $10$ different Wilf-equivalence classes.

    So far, we have $16$ Wilf-equivalent classes. It is easy to check that the remaining POPs are trivially related to the ones mentioned so far. Therefore, there are a total of $16$ Wilf-equivalence classes of type I. Detailed classification of type I POPs is provided in Table \ref{table3}.

\begin{table}[!htbp]
    \centering
\begin{tabular}{|l|l|l|l|l|l|l|l|l|l|}
\noalign{\hrule height 1pt}
\textbf{No. }&\textbf{POP of Type I} $(p)$ & $|\Av_n(p)|~(n=1,2,\ldots,8,\ldots)$ &  \textbf{OEIS}\\
\noalign{\hrule height 1pt}
& $(1,2,3,4;5)$ & &\\
& $(2,1,4,3;5)$ & &\\
1 & $(2,1,3,4;5)$ &  $1,2,6, 24,115, 618, 3591, 22088,\ldots $ & \href{https://oeis.org/search?q=1%2C2%2C6%2C+24%2C115%2C+618%2C+3591%2C+22088&go=Search}{A128088}\\
& $(3,2,1,4;5)$ & &\\
& $(3,1,2,4;5)$ & &\\
\hline
2 & $(2,3,1,4;5)$ & $1,2,6,24,115, 618,3584, 21920, \ldots$ &\\
& $(3,1,4,2;5)$ & &\\
\hline
3 & $(1,3,2,4;5)$ &$ 1,2,6, 24, 115, 618, 3591, 22096, \ldots $ &\\

\noalign{\hrule height 1.5pt}
 &	$(5,3,2,1;4)$ &  &\\
&$(5,3,1,2;4)$ &  &\\
4 &$(5,1,2,3;4)$ &  $1,2,6,24,115, 619, 3612, 22386, \ldots$ &\\
&$(3,5,2,1;4)$ & &\\
&$(3,5,1,2;4)$ & &\\
\hline

 5& $(5,2,1,3;4)$& $1, 2, 6, 24, 115,618, 3592, 22102,\ldots$ &\\
 &$(5,1,3,2;4)$ & &\\

 \hline
6 & $(5,2,3,1;4)$& $1,2,6,24,115, 619, 3613, 22412, \ldots$&\\

\hline
7 & $(2,5,3,1;4)$& $1,2,6,24,115, 619, 3607, 22257, \ldots$&\\

\hline 
8 & $(2,5,1,3;4)$ &$1,2,6,24,115,618,3587, 22000, \ldots $ &\\
 \hline
9 & $(1,5,3,2;4)$& $1,2,6,24, 115, 619, 3608, 22293,\ldots$&\\

 \hline
10 & $(1,5,2,3;4)$& $1,2,6,24,115, 619, 3606, 22232 ,\ldots$&\\

\noalign{\hrule height 1.5pt}
& $(4,5,1,2;3)$ & &\\
11 &  $(5,4,2,1;3)$  &$1, 2, 6, 24, 115, 619, 3614, 22425, \dots$ &\\
 &$(5,4,1,2;3)$ & &\\
 
\hline
12 & $(5,2,4,1;3)$& $1,2,6,24, 115, 619, 3615, 22457,\ldots$&\\

\hline
13 & $(4,2,5,1;3)$& $1,2,6,24, 115, 619, 3608, 22272 ,\ldots$&\\
\hline
14 & $(2,5,4,1;3)$& $1,2,6,24, 115, 619, 3609, 22297,\ldots$&\\

\hline
15 & $(2,5,1,4;3)$& $1,2,6,24,  115, 619, 3601, 22147,\ldots$&\\
\hline

16 & $(2,4,5,1;3)$& $1,2,6,24, 115, 619, 3614, 22426,\ldots$&\\

\hline

\end{tabular}
\caption{Classification of Type I POPs of size $5$ }
\label{table3}
\end{table}
\begin{remark}
        This proves a conjecture of Dimitrov \cite[Conjecture 8.1]{stoyan-distant-pattern}, which was also presented at the problem session of the British Combinatorial Conference, 2024 \cite{cameron-bcc30}. 
The author conjectured the following.
\begin{enumerate}
    \item \label{conj1}$av_n(1\sq 234)=av_n(1\sq243)=av_n(2\sq143)$.
    \item \label{conj2} $av_n(12\sq34)=av_n(12\sq43)=av_n(21\sq43)$.
    \item \label{conj3} $av_n(123\sq4)=av_n(124\sq3)=av_n(214\sq3)$.
\end{enumerate}
 Using the notation of POPs, we have $1\sq 234=(5,4,3,1;2) $, $1\sq243=(4,5,3,1;2)$, $2\sq 143=(4,5,1,3;2)$. These are trivially related to $(5,3,2,1;4), (5,3,1,2;4), (3,5,1,2;4)$ respectively. These are in the same class and correspond to Row $4$ of Table \ref{table3}.
 
 Similarly, $12\sq34=(5,4,2,1;3)$, $12\sq43=(4,5,2,1;3)\sim (5,4,1,2;3)$, $21\sq43=(4,5,1,2;3)$. These are in the same class and correspond to Row $11$ of Table \ref{table3}.
 
 Finally, $123\sq4=(5,3,2,1;4)$, $124\sq3=(3,5,2,1;4)$, $214\sq3=(3,5,1,2;4)$. This corresponds to Row $4$ of Table \ref{table3}.  
\end{remark}
    
\subsection{Type II}
Let the POP \begin{tikzpicture}[baseline=(current bounding box)]\filldraw  (0,0) circle (2pt);
        \filldraw  (0,.8) circle (2pt);
        \filldraw  (0,-.8) circle (2pt);
        \filldraw  (.7,0) circle (2pt);
        \filldraw  (1.4,0) circle (2pt);
		
		\draw (0,0) -- (0,.8); 
        \draw (0,0) -- (0,-.8); 
        \node[] at (-.2,.8)  { $a$};
		\node[] at (-.2,0)  { $b$};
		\node[] at (-.2,-.8)  { $c$};
        \node[] at (.9,0)  { $d$};
        \node[] at (1.6,0)  { $e$};
\end{tikzpicture} be denoted by $(a,b,c; d;e)$.
There are $\binom{5}{2}=10$ possible values for $d$ and $e$, and each value corresponds to $3!=6$ different patterns. Hence, we have $60$ POPs of type II.  The complement and reverse of $(a,b,c;d,e)$ are $(c,b,a;d,e)$ and $(6-a,6-b,6-c;6-d,6-e)$ respectively. Therefore, POPs with $(d,e)=(4,1),(3,1),(2,1),(4,2),(3,2)$ are the reverse of POPs with $(d,e)=(2,5),(3,5),(4,5),(2,4),(3,4)$ respectively. As $(a,b,c;d;e)=(a,b,c;e;d)$, it is enough to consider the POPs with $(d,e)=(2,5),(3,5),(4,5),(2,4),(3,4)$.
      
      Theorem \ref{iso_ver_beginning_end} tells us that the POPs  $(a,b,c;d;e)$  where  $(d;e)=(4;5)$ and $(1;2)$ are Wilf-equivalent as all classical patterns of size $3$ are Wilf-equivalent. In Section \ref{len4_type2}, we proved that all POPs of the shape \begin{tikzpicture}[baseline=(current bounding box)]\filldraw  (0,0) circle (2pt);
        \filldraw  (0,.8) circle (2pt);
        \filldraw  (0,-.8) circle (2pt);
        \filldraw  (.7,0) circle (2pt);
		
		\draw (0,0) -- (0,.8); 
        \draw (0,0) -- (0,-.8); 
        \node[] at (-.2,.8)  { $a$};
		\node[] at (-.2,0)  { $b$};
		\node[] at (-.2,-.8)  { $c$};
        \node[] at (.9,0)  { $d$};
\end{tikzpicture} where $d=2$ or $3$ are Wilf-equivalent. By Theorem \ref{iso_ver_beginning_end}, we see that the POPs $(a,b,c;d;e)$ where $\{d,e\}=\{2,5\},\{3,5\}$ are Wilf-equivalent. 
 The POPs $(1,3,5;2;4), (1,5,3;2;4), (1,2,5;3;4)$ and $(1,5,2;3;4)$  are in $4$ different Wilf-equivalence classes. 

So far, we have $6$  Wilf-equivalent classes. It is easy to check that the remaining POPs are trivially related to POPs mentioned so far. Therefore, there are $6$ Wilf-equivalence classes of type II. A detailed classification of type II POPs is provided in Table \ref{table4}.

\begin{table}[!htbp]
    \centering
\begin{tabular}{|l|l|l|l|l|l|l|l|l|l|}
\noalign{\hrule height 1pt}
\textbf{No. }&\textbf{POP of Type II} $(p)$ & $|\Av_n(p)|~(n=1,2,\ldots,8,\ldots)$  \\
\noalign{\hrule height 1pt}

 &	$(1,2,3;4;5)$ &  \\
1 & $(1,3,2;4;5)$ &  $1,2,6,24,100, 420, 1764, 7392, \ldots$ \\
& $(3,1,2;4;5)$ &\\

\noalign{\hrule height 1.5pt}

 & $(1,3,4;2;5)$ &  \\
& $(1,4,3;2;5)$ &\\
& $(4,1,3;2;5)$ &\\
2 & $(1,2,4;3;5)$ & $1,2,6,24,100, 426, 1848, 8120, \ldots$  \\
& $(1,4,2;3,5)$ &\\
& $(4,1,2;3;5)$ &\\

\noalign{\hrule height 1.5pt}

3 & $(1,3,5;2;4)$& $1, 2, 6, 24, 100,434, 1934, 8828 \ldots$ \\
& $(5,3,1;2;4) $ &\\

 \hline
 & $(1,5,3;2;4)$& \\
4 & $(3,1,5;2;4)$ & $1,2,6,24,100, 430, 1889, 8494  \ldots$\\
& $(5,1,3;2;4)$ &\\

\noalign{\hrule height 1.5pt}

5 & $(1,2,5;3;4)$& $1,2,6,24,100, 426, 1875, 8482,  \ldots$\\
& $(5,1,2;3;4)$ &\\

 \hline
6 & $(1,5,2;3;4)$& $1,2,6,24, 100, 426, 1855, 8278 ,\ldots$\\

\hline

\end{tabular}
\caption{Classification of Type II POPs of size $5$}
\label{table4}
\end{table}

\subsection{Type III}
Let the POP  \begin{tikzpicture}[baseline=(current bounding box)]
		\filldraw  (0.8,-.4) circle (2pt); 
		\filldraw   (0.8,.4) circle (2pt);
		\filldraw  (0,0) circle (2pt);
        \filldraw  (0,.8) circle (2pt);
        \filldraw  (0,-.8) circle (2pt);
		
		\draw (0.8,-.4) -- (0.8,.4); 
        \draw (0,0) -- (0,.8); 
        \draw (0,0) -- (0,-.8); 
		\node[] at (-.2,.8){ $a$};
        \node[] at (-.2,-.8){$c$};
        \node[] at (-.2,0) { $b$};
        \node[] at (1,.5)  { $d$};
		\node[] at (1,-.5) { $e$};
		
	\end{tikzpicture} be denoted by $(a,b,c;d,e)$. There are $\binom{5}{2}\times 2=20$ possible values for $(d,e)$, and each value corresponds to $3!=6$ different patterns. So, we have $120$ POPs of type III.  The complement and reverse of $(a,b,c;d,e)$ are $(c,b,a;e,d)$ and $(6-a,6-b,6-c;6-d,6-e)$ respectively. Therefore, the POPs with $(d,e)=(5,4),(5,3),(5,2),(5,1)$, $(4,3),(4,2)$ are the reverse of the POPs with $(d,e)=(1,2),(1,3),(1,4),(1,5),(2,3),(2,4)$ respectively.  Therefore, it is enough to consider the POPs with $(d,e)=(1,2),(1,3),(1,4),(1,5)$, $(2,3),(2,4)$.

When $\{a,b,c\}=\{1,2,3\}$ and $\{d,e\}=\{4,5\}$ or $\{a,b,c\}=\{3,4,5\}$ and $\{d,e\}=\{1,2\}$, all POPs are Wilf-equivalent by Theorems \ref{thm: disjoint-sum-shape-w} and \ref{thm: disjoint-sum-wilf}.    
Under trivial bijections, the following POPs are in $18$ ($6$ from each) different equivalence classes.
    \begin{enumerate}
        \item  $\{a,b,c\}=\{1,2,4\}$ and $\{d,e\}=\{3,5\}$ or $\{a,b,c\}=\{2,4,5\}$ and $\{d,e\}=\{1,3\}$.
        \item $\{a,b,c\}=\{1,3,4\}$ and $\{d,e\}=\{2,5\}$ or $\{a,b,c\}=\{2,3,5\}$ and $\{d,e\}=\{1,4\}$.
        \item $\{a,b,c\}=\{1,2,5\}$ and $\{d,e\}=\{3,4\}$ or $\{a,b,c\}=\{1,4,5\}$ and $\{d,e\}=\{2,3\}$.
    \end{enumerate}
Similarly, under trivial bijections, the following POPs are in $8$ ($4$ from each) different Wilf-equivalence classes.
    \begin{enumerate}
        \item  $\{a,b,c\}=\{1,3,5\}$ and $\{d,e\}=\{2,4\}$.
        \item $\{a,b,c\}=\{2,3,4\}$ and $\{d,e\}=\{1,5\}$.
    \end{enumerate}
Therefore, there are $27$ Wilf-equivalence classes of type III. A detailed classification of type III POPs is provided in Table \ref{table5}.

\begin{table}[!htbp]
    \centering
\begin{tabular}{|l|l|l|l|l|l|l|l|l|l|}
\noalign{\hrule height 1pt}
\textbf{No. }&\textbf{POP of Type III} $(p)$ & $|\Av_n(p)|~(n=1,2,\ldots,8,\ldots)$  \\
\noalign{\hrule height 1pt}
 &	$(3,4,5;1,2)$ &  \\
1 & $(3,5,4;1,2)$ & $1,2,6,24, 110, 530, 2597 ,12796, \ldots$ \\
& $(5,3,4;1,2)$ &\\
\noalign{\hrule height 1.5pt}
 
2 & $(2,4,5;1,3)$ & $1,2,6,24,110,532,2629,13135,\ldots$ \\
\hline

3& $(2,5,4;1,3)$ & $1,2,6,24,110,532,2632,13188,\ldots$ \\
\hline
4& $(4,2,5;1,3)$ & $1,2,6,24,110,532,2628,13095,\ldots$ \\
\hline

5& $(5,4,2;1,3)$ & $1,2,6,24,110,533,2658,13527,\ldots$ \\
\hline

6& $(4,5,2;1,3)$ & $1,2,6,24,110,532,2638,13329,\ldots$ \\
\hline

7& $(5,2,4;1,3)$ & $1,2,6,24,110,533,2640,13195,\ldots$ \\

\noalign{\hrule height 1.5pt}

8& $(2,3,5;1,4)$ & $1,2,6,24,110,535,2679, 13632, \ldots$ \\
\hline 
9& $(2,5,3;1,4)$ & $1,2,6,24,110,535,2690, 13836, \ldots$ \\
\hline 

10 & $(3,2,5;1,4)$ & $1,2,6,24,110,531,2613,12974, \ldots$ \\
\hline 

11 & $(5,3,2;1,4)$ & $1,2,6,24,110,531,2601,12817,\ldots$ \\
\hline 

12 & $(3,5,2;1,4)$ & $1,2,6,24,110,531,2626,13192, \ldots$ \\
\hline 

13 & $(5,2,3;1,4)$ & $1,2,6,24,110,534,2666, 13534, \ldots$ \\

\noalign{\hrule height 1.5pt}

14 & $(2,3,4;1,5)$ & $1,2,6,24,110,536,2690, 13711, \ldots$ \\
\hline 

15 & $(2,4,3;1,5)$ & $1,2,6,24,110,530,2595,12759,\ldots$ \\
\hline 

16 & $(4,3,2;1,5)$ & $1,2,6,24,110,530,2564, 12190, \ldots$ \\
\hline 

17 & $(3,4,2;1,5)$ & $1,2,6,24,110,530,2575,12407,\ldots$ \\

\noalign{\hrule height 1.5pt}

18 & $(1,4,5;2,3)$ & $1,2,6,24,110,533,2663, 13637, \ldots$ \\
\hline 

19 & $(1,5,4;2,3)$ & $1,2,6,24,110,534,2678, 13748, \ldots$ \\
\hline 

20 & $(4,1,5;2,3)$ & $1,2,6,24,110,530,2607, 12997, \ldots$ \\
\hline 

21 & $(5,4,1;2,3)$ & $1,2,6,24,110,530,2605,  12996, \ldots$ \\
\hline 

22 & $(4,5,1;2,3)$ & $1,2,6,24,110,530,2617,13202, \ldots$ \\
\hline 

23 & $(5,1,4;2,3)$ & $1,2,6,24,110,533,2633, 13156, \ldots$ \\

\noalign{\hrule height 1.5pt}

24 & $(1,3,5;2,4)$ & $1,2,6,24,110,537,2727, 14261, \ldots$ \\
\hline

25 & $(1,5,3;2,4)$ & $1,2,6,24,110,533,2673,13757,\ldots$ \\
\hline

26 & $(5,3,1;2,4)$ & $1,2,6,24,110,533,2644, 13319, \ldots$ \\
\hline

27 & $(3,5,1;2,4)$ & $1,2,6,24,110,532,2628,13175,\ldots$ \\
\hline

\end{tabular}
\caption{Classification of Type III POPs of size $5$ }
\label{table5}
\end{table}

\subsection{Type IV}
Let the POP \begin{tikzpicture}[baseline=(current bounding box)]
		\filldraw  (0,-.4) circle (2pt); 
		\filldraw   (0,.4) circle (2pt);
		\filldraw  (.7,0.4) circle (2pt);
        \filldraw  (.7,-.4) circle (2pt);
        \filldraw  (1.4,0) circle (2pt);
		
		\draw (0,-.4) -- (0,.4); 
        \draw (.7,0) -- (.7,.4); 
        \draw (.7,0) -- (.7,-.4); 
		\node[] at (-.2,.5)  { $a$};
		\node[] at (-.2,-.5)  { $b$};
		\node[] at (.9,.5)  { $c$};
        \node[] at (.9,-.5)  { $d$};
        \node[] at (1.6,0)  { $e$};
		
	\end{tikzpicture} be denoted by $(a,b;c,d;e)$. There are $5$ possible values for $e$, and each value corresponds to $\binom{4}{2}\times2=12$ many values of $(a,b)$ and $\binom{2}{2}\times2=2$ many values of $(c,d)$. So, we have $120$ POPs of type IV. The complement and reverse of $(a,b;c,d;e)$ are $(b,a;d,c;e)$ and $(6-a,6-b;6-c,6-d;6-e)$ respectively. Therefore, POPs with $e=1,2$ are the reverse of POPs with $e=5,4$ respectively. As $(a,b;c,d;e)=(c,d;a,b;e)$, it is enough to consider the POPs with $e=3,4,5$ and $(a,b)=(1,2),(1,3),(1,4),(1,5)$. 

In Section \ref{len4_type3}, we proved that there are $5$ Wilf-equivalence classes of the shape \begin{tikzpicture}[baseline=(current bounding box)]
		\filldraw  (0,-.4) circle (2pt); 
		\filldraw   (0,.4) circle (2pt);
		\filldraw  (.7,0.4) circle (2pt);
        \filldraw  (.7,-.4) circle (2pt);
		
		\draw (0,-.4) -- (0,.4); 
        \draw (.7,0) -- (.7,.4); 
        \draw (.7,0) -- (.7,-.4); 
		\node[] at (-.2,.5)  { $a$};
		\node[] at (-.2,-.5) { $b$};
		\node[] at (.9,.5)    {$c$};
        \node[] at (.9,-.5)  { $d$};
		
	\end{tikzpicture}. By Theorem \ref{iso_ver_beginning_end}, there are $5$ Wilf-equivalence classes of   POPs $(a,b;c,d;e)$ where $e=1$ or $5$. By Theorem \ref{thm: disjoint-sum-wilf}, we see that$(1,2;3,4;5)\sim (1,2;4,5;3)$ and $(1,2;4,3;5)\sim (1,2;5,4;3)$. By Theorem \ref{thm: disjoint-sum-shape-w}, we get  $(1,2;3,5;4)\sim (1,2;5,3;4)$ . 
 The POPs $(1,4;2,5;3), (1,4;5,2;3)$, $(1,5;2,4;3)$, $(1,5;4,2;3)$, $(1,3;2,5;4)$, $(1,3;5,2;4)$, $(1,5;2,3;4)$, $(1,5;3,2;4)$ are in $8$ different Wilf-equivalence classes. 

So far, we have $14$ Wilf-equivalence classes. It is easy to check that the remaining POPs are trivially related to the POPs mentioned so far, giving us $14$ Wilf-equivalence classes of type IV. A detailed classification of type IV POPs is provided in Table \ref{table6}.

\begin{table}[!htbp]
    \centering
\begin{tabular}{|l|l|l|l|l|l|l|l|l|l|}
\noalign{\hrule height 1pt}
\textbf{No. }&\textbf{POP of Type IV} $(p)$ & $|\Av_n(p)|~(n=1,2,\ldots,8,\ldots)$  \\
\noalign{\hrule height 1pt}

 &	$(1,2;3,4;5)$ &   \\
1 &$(1,2;4,3;5)$ & $1,2,6,24, 90,300,910,2576, \ldots$\\
& $(1,2;4,5;3)$ &\\
 \hline
 
2 & $(1,3;4,2;5)$ & $1,2,6,24,90,312,1064, 3552,\ldots$  \\
\hline

3 & $(1,3;2,4;5)$& $1,2,6,24,90,312,1029, 3304,\ldots$ \\
 \hline

4 & $(1,4;3,2;5)$& $1, 2, 6, 24,90, 300, 938, 2864,  \ldots$ \\

\hline
5 & $(1,4;2,3;5)$ & $1,2,6, 24, 90, 318, 1092, 3680 \ldots$\\
\noalign{\hrule height 1.5pt}
6 & $(1,2;3,5;4)$ & $1,2,6,24, 90, 315, 1043, 3318,\ldots$ \\
 & $(1,2;5,3;4)$ &\\
 \hline
7 &$(1,3;2,5;4)$ & $1,2,6,24,90,311,1034,3401,\ldots$\\
\hline

8 &$(1,3;5,2;4)$ & $1,2,6,24,90,323,1129,3951,\ldots$\\
\hline

9 &$(1,5;2,3;4)$ & $1,2,6,24,90,310,1034,3458,\ldots$\\
\hline

10 &$(1,5;3,2;4)$ & $1,2,6,24,90,301,914,2768, \ldots$\\
\noalign{\hrule height 1.5pt}

11 &$(1,4;2,5;3)$ & $1,2,6,24,90,329,1192,4302,\ldots$\\
\hline

12 &$(1,4;5,2;3)$ & $1,2,6,24,90,310,1088,3888,\ldots$\\
\hline
13 & $(1,5;2,4;3)$ & $1,2,6,24, 90,334,1235,4567, \ldots$ \\
\hline
14 & $(1,5;4,2;3)$ & $1,2,6,24,90,316,1009,3210,\ldots$\\
\hline
\end{tabular}
\caption{Classification of Type IV POPs of size $5$}
\label{table6}
\end{table}

\bibliographystyle{acm}
\end{document}